\documentclass{amsart}
\usepackage{amsmath}

\usepackage{amssymb}
\usepackage{amsthm}
\usepackage{amscd,amsbsy}
\usepackage{mathtools}
\usepackage{xcolor}
\usepackage{tikz-cd}
\usetikzlibrary{arrows,matrix}
\usepackage{float}
\usepackage{mathtools}
\usepackage{enumitem}
\usepackage{mathrsfs}  

\newtheorem{theorem}{Theorem}[section]
\newtheorem{lemma}[theorem]{Lemma}

\newtheorem{corollary}[theorem]{Corollary}

\newtheorem{addendum}[theorem]{Addendum}

\theoremstyle{definition}
\newtheorem{definition}[theorem]{Definition}

\newtheorem{example}[theorem]{Example}

\newtheorem{definitions and remarks}[theorem]{Definitions and Remarks}

\theoremstyle{remark}

\newtheorem{remark}[theorem]{Remark}
\newtheorem{remarks}[theorem]{Remarks}

\numberwithin{equation}{section}

\newcommand{\inv}{\mathrm{inv}}
\newcommand{\ATWinv}{\mathrm{ATWinv}}

\newcommand{\ord}{\mathrm{ord}}

\newcommand{\codim}{\mathrm{codim}\,}

\newcommand{\lcm}{\mathrm{lcm}}

\newcommand{\nc}{\mathrm{nc}}

\newcommand{\quot}{\mathrm{quot}}
\newcommand{\cp}{\mathrm{cp}}

\newcommand{\Spec}{\mathrm{Spec}}
\newcommand{\Specan}{\mathrm{Specan}}
\newcommand{\SL}{\mathrm{SL}}
\newcommand{\SLplus}{\mathrm{SL}_{\mathrm{lex}}^{+}}

\newcommand{\Hom}{\mathrm{Hom}}

\newcommand{\al}{{\alpha}}
\newcommand{\be}{{\beta}}
\newcommand{\de}{{\delta}}
\newcommand{\ep}{{\varepsilon}}
\newcommand{\De}{{\Delta}}
\newcommand{\ga}{{\gamma}}
\newcommand{\Ga}{{\Gamma}}

\newcommand{\la}{{\lambda}}

\newcommand{\s}{{\sigma}}

\newcommand{\vp}{{\varphi}}

\newcommand{\om}{{\omega}}

\newcommand{\IN}{{\mathbb N}}
\newcommand{\IQ}{{\mathbb Q}}
\newcommand{\IA}{{\mathbb A}}
\newcommand{\IP}{{\mathbb P}}

\newcommand{\IC}{{\mathbb C}}
\newcommand{\IK}{{\mathbb K}}
\newcommand{\IL}{{\mathbb L}}

\newcommand{\IZ}{{\mathbb Z}}

\newcommand{\cC}{{\mathcal C}}

\newcommand{\cO}{{\mathcal O}}

\newcommand{\cS}{{\mathcal S}}

\newcommand{\oj}{\overline{j}}

\newcommand{\oG}{\overline{G}}
\newcommand{\oF}{\overline{F}}
\newcommand{\obe}{\overline{\be}}
\newcommand{\oal}{\overline{\al}}
\newcommand{\oell}{\overline{\ell}}

\newcommand{\tX}{{\widetilde X}}

\newcommand{\tD}{{\widetilde D}}
\newcommand{\tE}{{\widetilde E}}

\newcommand{\tS}{{\widetilde S}}

\newcommand{\tW}{{\widetilde W}}
\newcommand{\tZ}{{\widetilde Z}}

\newcommand{\ts}{{\widetilde \sigma}}

\newcommand{\hf}{{\widehat f}}

\newcommand{\hy}{{\hat y}}

\newcommand{\ucC}{\underline{\cC}}

\newcommand{\dx}{{\dot x}}

\newcommand{\llb}{{[\![}}
\newcommand{\rrb}{{]\!]}}
\newcommand{\llbr}{{(\!(}}
\newcommand{\rrbr}{{)\!)}}

\newcommand{\RN}[1]{%
  \textup{\uppercase\expandafter{\romannumeral#1}}%
}

\begin{document}
\title[Group-circulant singularities and partial desingularization]
{Group-circulant singularities and partial desingularization preserving normal crossings}
\author[A.~Belotto da Silva]{Andr\'e Belotto da Silva}
\author[E.~Bierstone]{Edward Bierstone}
\address[A.~Belotto da Silva]{Universit\'e Paris Cit\'e and Sorbonne Universit\'e, UFR de Math\'ematiques, Institut de Math\'ematiques 
de Jussieu-Paris Rive Gauche, UMR7586, F-75013 Paris, France, and Institut universitaire de France (IUF)}
\email{andre.belotto@imj-prg.fr}
\address[E.~Bierstone]{University of Toronto, Department of Mathematics, 40 St. George Street, Toronto, ON, Canada M5S 2E4}
\email{bierston@math.utoronto.ca}
\thanks{Research supported by Agence Nationale de la
Recherche (ANR) projects ANR-25-ERCC-0003-01, ANR-22-CE40-0014 (Belotto da Silva), and
NSERC Discovery Grant RGPIN-2017-06537 (Bierstone)}
\date{\today}


\begin{abstract}
The subject is partial desingularization preserving the normal crossings singularities of an algebraic or analytic variety $X$
(over $\IC$ or over an uncountable algebraically closed field $\IK$ of characteristic zero, in the algebraic case).
Our approach has three parts
involving distinct techniques: (1) a formal splitting theorem for regular or analytic functions which satisfy a generic splitting
hypothesis; (2) a study of singularities in the closure of the normal crossings locus, based on the combinatorics of 
$G$-circulant matrices, where $G$ is a finite abelian group, leading to a theorem on reduction to \emph{group-circulant normal form};
(3) a partial desingularization theorem, proved using (1) and (2) together with 
weighted blowings-up of group-circulant singularities. Previous results were for partial desingularization preserving
simple normal crossings, or preserving general normal crossings when $\dim X \leq 4$.
\end{abstract}

\maketitle

\setcounter{tocdepth}{1}
\tableofcontents

\section{Introduction}\label{sec:intro}

\subsection{Overview}\label{subsec:overview}
The subject of this article is partial desingularization preserving the normal crossings locus
of an algebraic or analytic variety $X$ (over $\IC$ or over an uncountable algebraically closed field $\IK$ of characteristic zero, 
in the algebraic case). 
Our approach has three parts involving distinct techniques, which we believe may be of independent interest. 

\smallskip\noindent
(1) A formal splitting or factorization theorem for regular or analytic functions which satisfy a generic splitting
hypothesis (Theorem \ref{thm:splitintro}). This can be viewed as a multivariate version of the Newton-Puiseux theorem (cf. Macdonald \cite{McD}, 
Soto-Vicente \cite{SV}) or as a variant of the Abhyankar-Jung theorem (cf. Parusi\'nski-Rond \cite{PR}).

\smallskip\noindent
(2) A study of singularities in the closure of the normal crossings locus, generalizing the classical pinch point or Whitney umbrella,
based on the combinatorics of \emph{$G$-circulant matrices}, where $G$ is a finite abelian group. Theorem \ref{thm:ordpintro}
on reduction to \emph{group-circulant normal form} (Definition \ref{def:groupcircintro}) by smooth blowings-up plays a central part
in the paper.
According to Kanemitsu and Waldschmidt \cite{KW}, the idea of a group-circulant matrix is due to Dedekind \cite{Ded}.
The term ``$G$-circulant'' was introduced by Diaconis \cite[Chapt.\,3]{Diac}.

\smallskip\noindent
(3) A theorem on partial desingularization preserving normal crossings obtained using (1) and (2) together with 
weighted blowings-up of group-circulant singularities (Theorem \ref{thm:orb1pintro}). 
A \emph{normal crossings} singularity (of \emph{order} $k$)
is defined by a monomial equation $x_1\cdots x_k = 0$, where $x_1,\ldots,x_k$ are formal or analytic
coordinates, as opposed to the more restrictive notion of \emph{simple normal crossings}, where $x_1,\ldots,x_k$ form part of
a system of regular parameters. (A normal crossings singularity at a given point of an algebraic variety $X$ is simple
normal crossings in an \'etale neighbourhood.) Techniques of Abramovich, Temkin and W{\l}odarczyk \cite{ATW19}, \cite{Wlodar}
are used to show that the weighted blowings-up needed are globally well-defined. W{\l}odarczyk has informed us that he has obtained Theorem \ref{thm:orb1pintro} (1)-(4), by a different approach.

\smallskip
For simple normal crossings, there is a proper birational (or bimeromorphic) morphism $\s: X' \to X$ given by a sequence of
smooth blowings-up preserving the simple normal crossings locus of $X$, such that $X'$ has only simple normal crossings
singularities; see \cite{BDMV}, \cite[Section 12]{BMinv}, \cite[Section 3]{BMmin1}, \cite{Kolog}, \cite{Sz}, and also Remark \ref{rem:snc}.

On the other hand, for normal crossings in general, already for $\dim X =2$, the best that can be achieved is a smooth
blow-up sequence $\s: X' \to X$ preserving the normal crossings locus of $X$, such that $X'$ has only normal crossings
or pinch point singularities \cite{BMmin1}; in particular, pinch points cannot be eliminated without modifying the $\nc$ (normal
crossings) locus. The question of analogous results in higher dimension was raised by
Koll\'ar \cite{Kolog}. Complete lists of singularities which must be admitted in $X'$ were given for $\dim X \leq 3$ in \cite{BLMmin2}
and, more recently, for $\dim X \leq 4$ in \cite{BBR}. In these dimensions, only standard circulant singularities, defined by 
$G$-circulant matrices with $G$ a finite cyclic group, appear in the closure in $X'$ of the normal crossings
locus of $X$. 

Note that, since circulant singularities cannot be eliminated from $X'$ in the preceding results, nor can any singularities
in a small neighbourhood of a circulant point. In dimension $\leq 3$, a complete list of singularities that must be admitted
after a birational (or bimeromorphic) morphism preserving the normal crossings locus of $X$ is given by circulant singularities
and their neighbours. Dimension $4$, however, involves a new phenomenon---an additional singularity has to be admitted
as a limit of singularities occuring in a neighbourhood of certain circulant points.

In this article, we prove that (in any given dimension) there is a
proper birational (or bimeromorphic) morphism $\s: X' \to X$ preserving the normal crossings locus of $X$, where $X'$ admits
only group-circulant singularities and their neighbours, together with a finite list of possible additional singularities
that can be described in a precise way using weighted blowings-up of group-circulant loci. The morphism $\s$ is induced
by a finite sequence of weighted blowings-up (over any relatively compact open subset, in the analytic case).

We give formulas for the latter in Theorem \ref{thm:invtnc} that are not as explicit as those of \cite{BBR}, 
which are obtained using blowings-up that are
combinatorial in nature to move away singularities which occur in a \emph{distinguished divisor} at a
circulant point. Already for $\dim X  = 5$, the combinatorial \emph{moving away} techniques of \cite{BBR} are
insufficient (in Example \ref{ex:nonirred}, for instance). The possibility of more explicit formulas in higher dimension
remains an interesting problem.

The three main parts of the article are described more precisely in the following subsections of the introduction, and developed
in detail in the Sections 2-4 below. Our overall strategy for partial desingularization is similar to that of \cite{BBR}. Special
cases of (1) and (2) above are used in \cite{BBR}, whereas (3) essentially replaces the moving away techniques of \cite{BBR}.
Theorems \ref{thm:splitintro} and \ref{thm:ordpintro} were announced in \cite[Rmk.\,1.23]{BBR}.

We will usually assume that the ground field $\IK$ is $\IC$, though all 
results for algebraic varieties hold over any uncountable algebraically closed field $\IK$ of characteristic zero. 
We have tried to make the article readable
on its own, but nevertheless refer to \cite{BBR} for certain arguments, and also refer to \cite{BBR} or to the \emph{Crash
course on the desingularization invariant} \cite[Appendix A]{BMmin1}
for an explanation of certain techniques from resolution of singularities that we use.

\subsection{The splitting theorem}\label{subsec:split}
Let $X \subset Z$ denote an embedded algebraic (or analytic) hypersurface, over an algebraically closed field $\IK$
of characteristic zero ($\IK = \IC$ in the analytic case). Let $S$ denote a smooth closed subvariety of $X$ of codimension $k$
in $Z$. If $X$ is $\nc(k)$ (i.e., normal crossings of order $k$) generically on $S$,
then the set of non-$\nc(k)$ points
of $X$ in $S$ is a proper closed algebraic (or analytic) subset (cf. \cite[Lemma 3.7]{BBR}), which can be transformed to $S\cap E$,
where $E \subset Z$ is an ordered simple normal crossings divisor transverse to $S$, by resolution of singularities. 

Then, given $a_0 \in S$, there is an \'etale (or analytic) neighbourhood of $a_0$ in $Z$ with coordinates
\begin{equation*}
(w,u,x,z) = (w_1,\ldots,w_r, u_1,\ldots,u_q, x_1,\ldots,x_{k-1},z)
\end{equation*}
in which $\{w_j=0\}$, $j=1,\ldots,r$, are the components of $E$ at $a_0=0$, $S = \{z=x=0\}$, and the ideal of $X$ is generated
by a function
\begin{equation}\label{eq:weierpoly}
f(w,u,x,z) = z^k + a_1(w,u,x)z^{k-1} +\cdots + a_k(w,u,x),
\end{equation}
where the coefficients $a_i(w,u,x)$ are regular (or analytic) functions. The hypotheses imply that
$f$ is in the ideal generated by $x_1,\ldots,x_{k-1}, z$,
and $f$ splits formally (into $k$ factors of degree $1$ in $z$) at every point where $z=x=0,\, w_1\cdots w_r \neq 0$.

\begin{theorem}\label{thm:splitintro} 
Let $f(w,u,x,z)$ denote a function \eqref{eq:weierpoly}, where the coefficients $a_i(w,u,x)$ are regular (or analytic) functions,
$f$ is in the ideal generated by $x_1,\ldots,\allowbreak x_{k-1}, z$,
and $f$ splits formally at every point where $z=x=0,\, w_1\cdots w_r \neq 0$.
Assume that $\IK$ is uncountable.
Then, after finitely many blowings-up with successive centres of the form 
\begin{equation}\label{eq:splitcentre}
\{z=x=w_j =0\},\quad 1 \leq j \leq r,
\end{equation}
we can assume that $f$ splits over $\IK\llb w^{1/p},u,  x\rrb$, for some positive integer $p$,
where $w^{1/p} := (w_1^{1/p},\ldots, w_r^{1/p})$;
i.e., that 
\begin{equation}\label{eq:splitintro}
f(w,u,x,z) = \prod_{i=1}^k \left(z + b_i (w^{1/p}, u, x)\right),
\end{equation}
where each $ b_i (w^{1/p}, u, x) \in \IK\llb w^{1/p},u,  x\rrb$.
\end{theorem}

Theorem \ref{thm:splitintro} is proved in \cite[Thm.\,1.19]{BBR} in the special case that $q=0$ and $r=1$.

\begin{remark}\label{rem:splitintro}
Let $S'$ denote the strict transform of $S := \{z=x=0\}$ by a blowing-up $\s$ with centre of the form \eqref{eq:splitcentre}.
Then $S' \cong S$; moreover, $S' = \{z=x=0\}$ again, in the $w_j$-coordinate chart of $\s$, and $a_0=0$ lifts to $a_0'=0$ in $S'$.
The meaning of the statement in the theorem is that, after finitely many blowings-up with centres of the form \eqref{eq:splitcentre},
the strict transform $X'$ of $X$ is given at $a_0'=0$ by a function $f'$ with a splittting as in \eqref{eq:splitintro}.
\end{remark}

\begin{example}\label{ex:basic} (See also \cite[Example 1.18]{BBR}.) Let
$$
f(w,x,z) = z^2 + (w^3 + x) x^2.
$$
Then $f$ (or the subvariety $X$ of $\IA_{\IC}^3$ defined by $f(w,x,z)=0$) is $\nc(2)$ at
every nonzero point of the $w$-axis $\{x=z=0\}$. The function $f$ does not split over $\IC\llb w,x\rrb$,
but we can write
$$
f(v^2,x,z) = z^2 + v^6\left(1+\frac{x}{v^6}\right) x^2,
$$
so that $f(w,x,z)$ splits in $\IC(w^{1/2})\llb x\rrb [z]$, where $\IC(w^{1/2})$ denotes the field of fractions of $\IC[w^{1/2}]$. 

Consider the blowing-up $\s$ of the origin in $\IA_{\IC}^3$. The $w$-axis lifts to the $w$-chart of $\s$,
given by substituting $(w,wx,wz)$ for $(w,x,z)$, and the strict transform of $X$ is 
given by $f' = 0$ in the $w$-chart, where
$$
f'(w,x,z) := w^{-2} f(w,wx,wz) = z^2 + w(w^2 + x)x^2.
$$
After two more blowings-up of the origin, we get
$$
f'(w,x,z) = z^2 + w^3(1 + x)x^2, 
$$
so that $f'(w,x,z)$ splits over $\IC\llb w^{1/2},x\rrb$ (or $f'(v^2,x,z)$ splits in an \'etale neighbourhood of the origin).
\end{example}

\subsection{Group-circulant normal form}\label{subsec:Gcirc}
Let $G$ denote a finite abelian group. Let $g_1,\ldots,g_t$ denote any enumeration of the elements of $G$
(where $t$ is the order $|G|$ of $G$), and let $X_{g_j}$, $j=1,\ldots,t$, denote indeterminates. We define
the \emph{$G$-circulant matrix}
\begin{equation}\label{eq:Gcirc}
C_G(X_{g_1},\ldots,X_{g_t})
\end{equation}
as the $t\times t$ matrix whose $ij$'th entry is $X_{g_jg_i^{-1}}$. The matrix $C_G$ depends on the ordering of
the elements of $G$.

On the other hand, let
\begin{equation*}\label{eq:Gcircdet}
\De_{G} := \det C_{G}\,;
\end{equation*}
then both $\De_{G}$ and the eigenvalues of $C_{G}$ are independent of the 
ordering $g_1,\ldots,g_t$,
since the $i$'th and $j$'th rows (respectively, the $i$'th and $j$'th columns) of a $t \times t$ matrix are interchanged
by left multiplication (respectively, right multiplication) by the $t \times t$ matrix $E_{ij} = (a_{\al\be})$, where
$a_{ij} = a_{ji} = 1$, $a_{\al\al} = 1$ if $\al \neq i,j$, and $a_{\al\be} =0$ otherwise.

The eigenvalues of $C_G$ are linear combinations of the indeterminates $X_{g_1},\ldots,X_{g_t}$ (see 
\eqref{eq:eigenval} below and \eqref{eq:eigengen}).

\begin{example}\label{ex:standardcirc}
If $G = \IZ_k$, the (additive) cyclic group of order $k$, then $C_{\IZ_k} = C_k$ is the \emph{(standard)
circulant matrix}
\begin{equation}\label{eq:circ}
C_k(X_0,X_1,\ldots,X_{k-1}) :=
\begin{pmatrix}
X_0 & X_1 & \cdots & X_{k-1} \\
X_{k-1} & X_0 & \cdots & X_{k-2} \\
\vdots & \vdots & \ddots & \vdots \\
X_1 & X_2 & \cdots & X_0
\end{pmatrix}.
\end{equation}
Phillip Davis's book \cite{Davis} is a classic reference on circulant matrices.

The circulant matrix $C_k(X_0,X_1,\ldots,X_{k-1})$ has eigenvectors
$$
V_\ell = (1,\, \ep^\ell,\, \ep^{2\ell},\, \ldots,\, \ep^{(k-1)\ell}),
$$
$\ell = 0, \ldots, k-1$, where $\ep= e^{2\pi i/k}$. The corresponding
eigenvalues are
\begin{equation}\label{eq:eigenval}
Y_\ell = X_0 + \ep^\ell X_1 + \cdots + \ep^{(k-1)\ell}X_{k-1},\quad \ell = 0, \ldots, k-1.
\end{equation}

Let $\De_k = \De_{\IZ_k}$ denote the determinant $\det C_k$. Then
\begin{equation}\label{eq:detcirc}
\begin{aligned}
\De_k (X_0, \ldots, X_{k-1}) &=  Y_0 \cdots Y_{k-1}\\
               &= \prod_{\ell = 0}^{k-1}\, (X_0 + \ep^\ell X_1 + \cdots + \ep^{(k-1)\ell}X_{k-1}).
\end{aligned}
\end{equation}

Given indeterminates $(w,x_0,\ldots,x_{k-1})$, set
\begin{align}\label{eq:circfact}
P_k (w,x_0,\ldots,x_{k-1}) :=&\, \De_k (x_0,\, w^{1/k}x_1,\, \ldots,\, w^{(k-1)/k}x_{k-1})\\
                               =&\, \prod_{\ell = 0}^{k-1}\, (x_0 + \ep^\ell w^{1/k} x_1 + 
                                     \cdots + \ep^{(k-1)\ell}w^{(k-1)/k}x_{k-1})\nonumber
\end{align}
Then $P_k (w,x_0,\ldots,x_{k-1})$ is an irreducible polynomial.
Following \cite{BBR}, we define the \emph{circulant} (or \emph{circulant point}) \emph{singularity} $\cp(k)$ as the
singularity at the origin of the variety $X \subset Z = \IA^{k+1}$ defined by the equation
$P_k (w,x_0,\ldots,x_{k-1}) = 0$; i.e., by the equation
$$
\De_k (x_0,\, w^{1/k}x_1,\, \ldots,\, w^{(k-1)/k}x_{k-1}) = 0.
$$

For example, $\cp(2)$ is the pinch point, given by $P_2(w,z,x) = z^2 - wx^2$,
and $\cp(3)$ is given by
\begin{equation*}\label{eq:cp3expansion}
P_3 (w,z,y,x) = z^3 + wy^3 + w^2 x^3 - 3wxyz.
\end{equation*}

More generally, a \emph{product circulant singularity} $\cp(k_1)\times\cdots\times\cp(k_m)$ is defined by a polynomial
\begin{equation*}
P(w,x) = P_1(w,x^1)\cdots P_m(w,x^m),
\end{equation*}
where $x=(x^1,\ldots,x^m)$, $x^\ell = (x^\ell_0,\ldots,x^\ell_{k_\ell -1})$, $\ell = 1,\ldots,m$, and, for each $\ell$,
\begin{equation*}
P_\ell(w,x^\ell) = \De_{k_\ell} (x^\ell_0,\, w^{1/k_\ell}x^\ell_1,\, \ldots,\, w^{(k_\ell-1)/k_\ell}x^\ell_{k_\ell-1}),
\end{equation*}
(with a single variable $w$ in common).
\end{example}

\smallskip
\begin{definition}\label{def:groupcircintro}
\emph{Group-circulant singularity.}
We define a \emph{group-circulant} or \emph{$\Ga$-circulant singularity}, of \emph{order} $k$,
depending on a finite abelian group $\Ga$ and a positive integer $r$, as the singularity
at the origin of an irreducible polynomial
\begin{equation}\label{eq:normformcoords}
P(w,x) = P(w_1,\ldots,w_r,x_0,\ldots, x_{k-1}),
\end{equation}
where
\begin{equation}\label{eq:normformintro}
P(w,x) = \De_\Ga(X_{\ell_0}, X_{\ell_1},\ldots, X_{\ell_{k-1}})
\end{equation}
and the conditions (1) and (2) following hold.
\begin{enumerate}
\item $\{0 = \ell_0, \ell_1, \ldots, \ell_{k-1}\}$ is an enumeration of the elements of $\Ga$, and
\begin{align*}\label{eq:normform1}
X_{\ell_0} &= x_0,\\
X_{\ell_j} &= w^{\ga_j} x_j,\quad j=1,\ldots, k-1,
\end{align*}
where $w^{\ga_j} = w_1^{\ga_{j1}}\cdots w_r^{\ga_{jr}}$ and, for each $i=1,\ldots,r$, there is a positive integer $p_i$ 
dividing $k$, such that 
\begin{equation*}
\ga_{ji} \in \frac{1}{p_i}\left\{0,\ldots,p_i-1\right\},\quad j=1,\ldots, k-1,
\end{equation*}
with $0\neq \ga_{ji} = \be_{ji}/p_i$ in lowest terms, for some $j$.
\end{enumerate}

\smallskip\noindent
Thus each $p_i$ is the smallest positive integer such that $\ga_{ji} \in \IZ/p_i$, $j=1,\ldots,k-1$.
Let $G$ denote the abelian group $\IZ_{p_1}\times\cdots\times\IZ_{p_r}$. Then $G$ acts on $\IC[w_1^{1/p_1},\ldots,w_r^{1/p_r},x]$
in an evident way. Since $P(w,x)$ is $G$-invariant and $P(w,x)$ is the product of the eigenvalues $Y_j$ of 
$C_{\Ga}(X_{\ell_0}, X_{\ell_1},\ldots, X_{\ell_{k-1}})$, there is an induced homomorphism from $G$ to the group of permutations
of $\{Y_j\}$, and $G$ acts transitively on the set of eigenvalues since $P$ is irreducible.

\smallskip
\begin{enumerate}
\item[(2)] $\Ga = G/H$, where $H$ denotes the stabilizer of any eigenvalue $Y_j$.
\end{enumerate}

\smallskip
More generally, a \emph{product group-circulant singularity} is defined by a polynomial $P(w,x) = P_1(w,x^1)\cdots P_s(w,x^s)$, 
where $x=(x^1,\ldots,x^s)$ and, for each $h=1,\ldots,s$, $x^h=(x^h_0,\ldots,x^h_{k_h-1})$,
$P_h$ is an irreducible polynomial $P_h(w,x^h)$, and $P_h$ defines a $\Ga_h$-circulant singularity as above, for suitable $\Ga_h$
and $r_h \leq r$ (i.e., certain variables $w_i$ may not appear in $P_h$, for given $h$). The \emph{order} $k := k_1 +\cdots + k_s$.
\end{definition}

\begin{remark}\label{rem:groupcircintro}
It follows from Definition \ref{def:groupcircintro} that, at every point of a \emph{codimension one stratum}
$T_i := \{x=w_i = 0,\, w_h\neq 0,\, h\neq i\}$ (i.e., codimension one in $S := \{x=0\}$), \eqref{eq:normformintro} induces
standard product circulant normal form $\cp(p_i)\times\cdots\times\cp(p_i)$ (with $k/p_i$ factors); in particular,
for each $j=0,\ldots,k-1$, $\{\ga_{ji}: j = 0,\ldots,k-1\}$ in Definition \ref{def:groupcircintro} comprises $k/p_i$ copies of $q/p_i$, for each $q=0,\ldots,p_i-1$
(where $\ga_{0i} = 0$, $i=1,\ldots,r$). See 
Theorem \ref{thm:codim1}.

It then also follows that, for any $q<r$, $\De_\Ga$ induces a product group-circulant singularity of order $k$ at any point
of a codimension $q$ stratum
$\{x =w_{i_1} =\cdots = w_{i_q} = 0\,\text{ and }\,w_i \neq 0,\,\text{ if }\,i\neq i_j,\, j=1,\ldots q\}$, where $1 \leq i_1 < \cdots < i_q \leq r$.
\end{remark}

The standard circulant singularity $\cp(k)$ is the unique $\IZ_k$-circulant singularity with $r=1$.
(The isomorphism class of) a group circulant singularity is not uniquely determined
by $\Ga$ and $r$, even in the case that $\Ga = \IZ_k$ if $r>1$; cf. Example \ref{ex:nonirred}).

For the general case (when the stabilizer subgroup $H$ is not necessarily trivial in Definition \ref{def:groupcircintro}), we develop combinatorial
techniques involving the Pontryagin dual $G^* := \Hom(G, \IQ/\IZ)$ of $G$ (see \cite{Pont}) to extend the
simple construction of Example \ref{ex:standardcirc} to group-circulant matrices, in general
(\S\S\ref{subsec:circcomb}, \ref{subsec:codim1combin}). The orthogonal complement of $H$ in $G^*$ is
(non-canonically) isomorphic to $G/H$. (The \emph{character group} $\Hom(G, \IC^*) \cong \Hom(G, \IQ/\IZ)$, 
for $G$ finite abelian.)

We say that an algebraic or analytic variety $X$ has a \emph{group-circulant singularity} at a given point $a$ if there is a local
embedding variety in which the ideal of $X$ at $a$ is generated by a group-circulant singularity as in Definition \ref{def:groupcircintro},
with respect to coordinates in a suitable \'etale or analytic neighbourhood (in particular, $X$ is irreducible at $a$).
If $X$ has a group-circulant singularity as in \eqref{eq:normformintro} at $a$, then we also say that the ideal of $X$ (or a generator
of the ideal) has \emph{$\Ga$--circulant normal form} \eqref{eq:normformintro} at $a$. Likewise for \emph{product group-circulant normal form}.

The following theorem on \emph{reduction to group-circulant normal form by smooth blowings-up preserving normal crossings}
shows how group-circulant singularities come into partial desingularization,
and also plays a key part in the inductive step (induction on $p \leq n+1$) in Theorem \ref{thm:orb1pintro}.

\begin{theorem}\label{thm:ordpintro}
Let $X$ denote a complex algebraic or analytic variey. Let $n := \dim X$ and let $p$ denote a positive integer 
$1\leq p\leq n+1$. Then
there is a morphism $\s_p: X_p \to X$ given by the composite of a finite sequence of admissible smooth blowings-up
(over any relatively compact open subset of $X$, in the analytic case)
preserving the locus $X^{\nc(p)}$ of normal crossings points of order $\leq p$, such that $X_p$ has only hypersurface
singularities and has
maximum order $\leq p$, the subset $S_p$ of points of order $p$ of $X_p$ is smooth and of dimension $n-p+1$ (unless empty),
and $X_p$ has (product) group-circulant singularities of order $p$ at every point of $S_p$ (in particular, $X_p$ is generically $\nc(p)$ on $S_p$).
Moreover,
\begin{enumerate}
\item if $E_p$ denotes the exceptional divisor of $\s_p$, then $E_p$ is transverse to $S_p$, $S_p\backslash E_p$ corresponds
to the $\nc(p)$-locus of $X$, and $X_p$ has a (product) group-circulant singularity at each point of $S_p \cap E_p$, with normal form
(as in Definition \ref{def:groupcircintro}) with the parameters $w_i$ corresponding to components of $E_p$;

\smallskip
\item if $G$ is a group acting on $X$, then the centre of every blowing-up is invariant under the action of $G$, so that the action lifts by every blowing-up, eventually to an action on $X_p$ commuting with $\s_p$.
\end{enumerate}
\end{theorem}

A blowing-up is called \emph{admissible} if its centre is smooth and snc with the exceptional divisor. 
Group circulant singularities cannot be eliminated without
modifying $X$ at normal crossings points.

Our proof of Theorem \ref{thm:ordpintro} uses the Splitting Theorem \ref{thm:splitintro} and the desingularization invariant of
\cite{BMinv} together with a \emph{cleaning} argument based on that of \cite[Thm.\,1.22]{BBR}; see Theorem \ref{thm:normform}. 

\begin{remark}\label{rem:snc}
A product group-circulant singularity of order $p$ at a point where $X_p$ has $p$ irreducible components is snc.
The earlier results on partial desingularization preserving snc (as in \cite{BDMV}, \cite{BMmin1}), therefore, are direct consequences
of Theorem \ref{thm:ordpintro}, using induction on $p$ (Theorem \ref{thm:orb1pintro} is not needed).
\end{remark}

\begin{remarks}\label{rem:ordpintro} Theorem \ref{thm:ordpintro} can be strengthened in the following ways that are important
to the inductive proof of our partial desingularization theorem \ref{thm:orb1pintro}.

\medskip\noindent
(1) We can also consider an embedded variety $X \hookrightarrow Z$ (where $Z$ is smooth) together with an ordered (reduced)
divisor $E \subset Z$ (i.e., the \emph{components} $E_i$ of $E = \sum_{i=1}^t E_i$ are reduced, but not necessarily irreducible).
If $\tau: Z' \to Z$ is a birational (or bimeromorphic) morphism, then the \emph{transform} $E' \subset Z'$ of $E$ denotes the ordered
divisor $E' = \sum_{i=1}^{t+1} E'_i$, where $E'_i$ denotes the strict transform of $E_i$, $i=1,\ldots,t$, and $E'_{t+1}$ is
the exceptional divisor of $\tau$.

If $X$ is an embedded variety together with an snc divisor $E$ which is transverse to $S_p$, then the theorem holds, where
\emph{admissible} means that the centres of blowing up are smooth and snc with the (transformed) divisor; in this case, moreover,
the conclusion holds with $S_p$ transverse to the final divisor $E_p$.

\medskip\noindent
(2) There is also a version of Theorem \ref{thm:ordpintro} for $X$ together with a divisor $E$, relative to normal crossings
points of order $p$ of the pair $(X,E)$, where the latter means nc of order $p$ of $X \cup E$. This version, in fact, reduces to
the theorem as stated for $X \cup E$.

\medskip\noindent
(3) In (1) and (2), if $X$ and all components of $E$ are invariant with respect to the action of a group $G$, then the
theorem holds with the additional condition that (the group lifts to every blowing-up and) the centre of every blowing-up involved
is $G$-invariant, so that the action lifts to an action on $X_p$ commuting with $\s_p$, and every component of $E_p$ is $G$-invariant.
\end{remarks}

\begin{example}\label{ex:gencirc}
Let $G$ denote the \emph{Klein four-group} $\IZ_2 \times \IZ_2$. Consider the polynomial
$P(w,x) = P(w_1,w_2,x_0,x_1,x_2,x_3)$ given by
\begin{equation*}
P(w,x) = \prod_{i,j = 0,1} \left(x_0 + w_1^{1/2}x_1 + w_2^{1/2}x_2 + w_1^{1/2}w_2^{1/2}x_3\right),
\end{equation*}
where $\ep_2 = e^{2\pi i/2}$. Then $X := V(P)$ is irreducible, and $X$ has a product circulant singularity $\cp(2)\times\cp(2)$ at any point of the 
codimension one stratum $T_i = \{x=w_i=0,\, w_{3-i}\neq 0\}$, $i=1,2$. 
Set
\begin{align*}
Y_{ij} &:= x_0 + \ep_2^i w_1^{1/2} x_1 + \ep_2^j w_2^{1/2} x_2 + \ep_2^i \ep_2^j w_1^{1/2} w_2^{1/2} x_3,\\
X_{k\ell} &:= \frac{1}{4} \sum_{k,\ell = 0,1} \ep_2^{ki - \ell j} Y_{ij},
\end{align*}
and reindex the latter using lexicographic order of pairs: $X_0 := X_{00},\,X_1 := X_{01},\,X_2 := X_{10},\,X_3 := X_{11}$.

Then the $G$-circulant matrix $C_G(X_{00},X_{01},X_{10},X_{11})$ has nested block-circulant form
\begin{equation*}
C_G (X_0,X_1,X_2,X_3) = \left(
\begin{matrix}
X_0 & X_1 & X_2 & X_3 \\
X_1 & X_0 & X_3 & X_2 \\
X_2 & X_3 & X_0 & X_1 \\
X_3 & X_2 & X_1 & X_0
\end{matrix}
\right)
\end{equation*}
(not a standard circulant matrix) and
\begin{equation}\label{eq:klein}
P(w,x) = \De_G (X_0,X_1,X_2,X_3)\\
            = (\det C_G) \left(x_0,\, w_1^{1/2}x_1,\, w_2^{1/2}x_2,\, w_1^{1/2} w_2^{1/2} x_3\right).
\end{equation} 
Nested block-circulant form for a $G$-circulant matrix is a general phenonemon that we do not
pursue in this article, but which is related to Theorems \ref{thm:codim1} and \ref{thm:invtnc}.  
The $G$-circulant singularity \eqref{eq:klein} has smooth normalization, as does $\cp(k)$.
\end{example}

\subsection{Partial desingularization}\label{subsec:partialdesing}
In this subsection, we state and briefly explain Theorem \ref{thm:orb1pintro} on partial desingularization, which will be
proved in Section \ref{sec:wtblup}. Our approach is guided by the lower-dimensional results described in the Overview 
\S\ref{subsec:overview}.

Theorem \ref{thm:orb1pintro} is proved recursively; the inductive step involves Theorem \ref{thm:ordpintro}
followed by a sequence of weighted blowings-up of the group circulant points of the stratum $S_p$. This sequence of
weighted blowings-up of $X_p$ produces an orbifold $\tX_p \to \tX_{p,\quot} =: X'_p$ and birational (or bimeromorphic) morphism
$\tau: X'_p \to X_p$. Here $\tX_p$ denotes the collection of affine orbifold covering charts, and $\tX_p \to \tX_{p,\quot}$
denotes the collection of morphisms from each chart. The exceptional divisor of the sequence of weighted blowings-up
is called a \emph{distinguished divisor} $\tD_p \subset \tX_p$ or $D_p \subset X'_p$ (cf. \S\ref{subsec:overview}).

In general, an \emph{orbifold} $Y\to Y_\quot$ is a collection $Y$ of \emph{covering charts} together with morphisms 
which cover a variety $Y_\quot$, where each covering chart is equipped with an
action of a finite abelian group, and the morphism from the chart to $Y_\quot$ is the quotient morphism followed by an \'etale 
mapping (or open embedding, in the analytic case). See also \S\ref{subsec:wtblup}.

A weighted blowing-up (or a suitable sequence of weighted blowings-up) of an algebraic or analytic variety $X$ determines
an orbifold $\tX \to \tX_\quot$ and a commutative diagram,
\begin{equation}\label{eq:orbtriangle}
\begin{tikzcd}
\tX \arrow{d} \arrow{rd}{\ts} &\\
{\hspace{-.8cm}}X' := \tX_{\quot} \arrow{r}{\s} & X\,,
\end{tikzcd}
\end{equation}
where $\s: X' \to X$ is a birational (or bimeromorphic) morphism, and the other arrows denote collections of
morphisms from the covering charts.

The techniques needed for this construction are presented in Section \ref{sec:wtblup}. Since we use only standard
blowings-up (weighted blowing-up with trivial weights) together with explicit weighted blowings-up of group-circulant
singularities, the construction is relatively simple, and the technology will be developed only in the generality needed.

For weighted blowings-up of group-circulant singularities, we will use results of Abramovich, Temkin and W{\l}odarczyk \cite{ATW19}
and, in particular, \cite{Wlodar} on cobordant
blowing-up along smooth weighted centres, to obtain a well-defined variety $X'$ with
an explicit birational morphism $\s: X' \to X$ in \eqref{eq:orbtriangle}, without recourse to stacks. 
The argument applies also to the complex-analytic case (over a relatively compact open set), essentially
using \cite[Ch.\,1]{AHV} on $\Specan$ (cf.\,\cite[\S1.6]{ABTW}).
In the analytic case, moreover, the result can also be proved using a direct argument for gluing the orbifold quotient charts; this
will be presented in a follow-up article.

Let us begin with an example to illustrate the way we use weighted blowings-up of group-circulant singularities in Theorem \ref{thm:orb1pintro},
For simplicity, we work with a standard (i.e., cyclic group) circulant singularity in this example, but the case of a general (product) group
circulant singularity will be handled in Section \ref{sec:wtblup} in a completely analogous way.

It will be convenient to make a small change in notation from \S\ref{subsec:Gcirc} when we work with weighted blowings-up
of group-circulant singularities, because of the way that the orbifold groups arise in a weighted blow-up: we will write the
cyclic factors of a group $G = \IZ_{p_1} \times \cdots \times \IZ_{p_r}$ as multiplicative cyclic groups $\mu_{p_i}$ rather than
additive groups $\IZ_{p_i}$, so we write $G = \mu_{p_1} \times \cdots \times \mu_{p_r}$ corresponding to the parameters
or coordinates $(w_1,\ldots,w_r)$ in Definition \ref{def:groupcircintro}.

\begin{example}\label{ex:orbcpkintro} We consider the cyclic circulant singularity $\cp(k)$ defined by $\De_k(x_0,w^{1/k}x_1,\ldots,w^{(k-1)/k}x_{k-1})$,
which we write in short form as $\De(w^{j/k}x_j)$ ($j=0,\ldots,k-1$). 
Write $\De = \De_k$.
Consider the weighted blowing-up $\s$
of the origin with weights $(k, k-j+1)$ associated to the parameters $(w, x_j)$.

The weighted blow-up is covered by $k+1$ affine charts, corresponding to the parameters $w,\,x_j$, $j=0,\ldots,k-1$.
In the orbifold $w$-chart, $\s$ is given by the substitution
\begin{align*}
w &= t^k,\\
x_j &= t^{k-j+1}\dx_j,\quad j=0,\ldots,k-1,
\end{align*}
and the group $\mu_k$ acts on the chart by
\begin{equation*}
\ep\cdot (t, \dx_j) := (\ep t, \ep^{-(k-j+1)}\dx_j) = (\ep t, \ep^{j-1}\dx_j), \quad \text{where } \ep \in \mu_k.
\end{equation*}
The group action is free outside the exceptional divisor $\{t=0\}$ (which will be denoted $\tD_k$ in the discussion
of Theorem \ref{thm:orb1pintro} below).

Likewise, the orbifold $x_j$-chart has an action of $\mu_{k-j+1}$, $j=0,\ldots,k-1$.

The strict transform by $\s$ of the stratum $S = \{x=0\}$ intersects only
the $w$-chart (as an invariant smooth subvariety), and
the pullback of $\cp(k)$ is given in this chart by $t^{k(k+1)}\De(\dx_0,\ldots,\dx_{k-1})$, which is normal crossings.

The $w$-chart has quotient given by the following \emph{Hilbert basis} for the invariants (i.e., set of generators 
for the algebra of $\mu_k$-invariant polynomials):
\begin{align*}
W &= t^k,\\
X_0 &= t\dx_0,\\
X_1 &= \dx_1,\\
X_j &= t^{k-(j+1)}\dx_j, \quad 2\leq j\leq k-1,\\
S_{\mu,\la} = S_{\mu,(\la_0,\ldots,\la_{k-1})} &= t^\mu\prod_{j=0}^{k-1} \dx_j^{\la_j},
\end{align*}
where $\mu$ and the $\la_j$ are nonnegative integers such that $\la_1 = 0$ and
\begin{equation*}
\mu + \la_0(k-1) + \sum_{j\geq 1}\la_j(j-1) = \nu k, \quad \nu = \nu_{\mu,\la} \geq 1.
\end{equation*}
(The latter invariants include those preceding.)

The image of $\{\De(\dx_0,\ldots,\dx_{k-1})=0\}$ by the preceding quotient morphism is given by
the ideal generated by 
\begin{equation}\label{eq:quotgen}
W^{-2}\De(WX_0, W^{1+1/k}X_1, W^{j/k}X_j, \, j\geq 2)
\end{equation}
(which is just $\cp(k)$, in standard form after a permutation of $(X_0,\ldots,X_{k-1})$),
together with the relations among the invariants, which include, in particular,
\begin{equation}\label{eq:relinvintro}
\begin{aligned}
\prod_{j\neq 1} X_j^{\la_j} &= t^{\la_0 + \sum_{j\geq 2}\la_j(k-j+1)}\prod_{j\neq 1} \dx_j^{\la_j}\\
&= t^{\la_0 k + \sum_{j\geq 2}\la_j k - (\mu + \la_0(k-1) + \sum_{j\geq 2} \la_j(j-1))} t^\mu \prod_{j\neq 1} \dx_j^{\la_j}\\
&= W^{\sum_{j\neq 1}\la_j - \nu} S_{\mu,\la}.
\end{aligned}
\end{equation}
Note that the Hilbert basis above includes $S_{(0)}=\dx_0^{\la_0}$, where $\la_0(k-1) = \nu_0 k,\, \nu_0\geq 1$,
$S_{(1)} = X_1 = \dx_1$, and
$S_{(j)} = \dx_j^{\la_j}$, where $\la_j(j-1) = \nu_j k,\, \nu_j\geq 1$, for each $j\geq 2$.
Therefore, the relations \eqref{eq:relinvintro} above include $X_0^{\la_0} = W^{\la_0 - \nu_0} S_{(0)}$ and
$X_j^{\la_j} = W^{\la_j -\nu_j}S_{(j)}$,
for each $j\geq 2$. Thus the quotient variety is the intersection of the circulant hypersurface \eqref{eq:quotgen} with the toric variety
defined by the ideal of relations among the invariants.

We will call this quotient an \emph{orbifold circulant singularity} corresponding to $\De(w^{j/k}x_j)$ ($\cp(k)$).
See Example \ref{ex:orbgencirc} for general orbifold group circulant singularities.

The induced morphism from the quotient to the original space is given by
\begin{align*}
w &= t^k = W,\\
x_0 &= t^{k+1}\dx_o = WX_0,\\
x_1 &= t^k \dx_1 = WX_1,\\
x_j &= t^{k-j+1}\dx_j = X_j,\quad j=2,\ldots,k-1.
\end{align*}
Under this morphism, $\De(w^{j/k}x_j)$ pulls back to \eqref{eq:quotgen}\,$\times W^2$.

\begin{remark}\label{rem:orbcpkintro}
Note that the \eqref{eq:relinvintro} can be rewritten as
\begin{equation*}
\frac{S_{\mu,\la}}{W} = W^{\nu -1} \prod_{j\neq 1}\left(\frac{X_j}{W}\right)^{\la_j},
\end{equation*}
which describes how the relation \eqref{eq:relinvintro} transforms by blowing up $\{(W, X_j, S_{\mu,\la})=0\}$
(in the $W$-chart).
The quotient variety is, therefore, transformed by this blowing-up  to the hypersurface $\cp(k)$ \eqref{eq:quotgen}
in the graph of (the set of functions) $S'_{\mu,\la} := S_{\mu,\la}/W$. 
The blowing-up itself is simple toric resolution of singularities of the ideal of relations; cf. \cite[Section 8]{BMtoric}.
\end{remark}
\end{example}

The partial desingularization theorem following involves a positive
integer $p$, as in Theorem \ref{thm:ordpintro}. We are mainly interested in the case $p = \dim X +1$ (so that the open subset $X^{\nc(p)}$ of normal
crossings points of order at most $p$ of $X$ coincides with the locus $X^{\nc}$ of all normal crossings points), but
the statement involving $p$ is useful for induction.

\begin{theorem}\label{thm:orb1pintro}
Let $X$ denote a complex algebraic or analytic variey, and let $p$ be a positive integer $\leq \dim X +1$.
Then there is an orbifold diagram \eqref{eq:orbtriangle},
given by a finite sequence of weighted blowings-up (over any relatively compact open subset of $X$,
in the analytic case), such that the following conditions hold.
\begin{enumerate}
\item Let $\tE$ denote the exceptional divisor of $\tX$ (the collection of exceptional divisors in each covering chart).
Then $\ts|_{\tX\backslash \tE}$ is \'etale over $X^{\nc(p)}$, and 
induces an isomorphism $X'\backslash E' \to X^{\nc(p)}$. 
(The orbifold structure is given by the action of a finite abelian group
in every orbifold coordinate chart of $\tX$, and the group action is free on $\ts^{-1}(X^{\nc(p)})$.)

\smallskip
\item $\tX$ has only normal crossings singularities, of order at most $p$.

\smallskip
\item The induced morphism $\s: X' = \tX_\quot \to X$ is proper, birational (or bimeromorphic), and preserves $X^{\nc(p)}$.

\smallskip
\item For each $q=1,\ldots,p$, let $\tS_q$ denote the subset of nc points of order $q$ of $\tX$. (The collection $\{\tS_q\}$ is a
stratification of $\tX$.) 
Then:
\begin{enumerate}

\smallskip
\item $\tE$ is transverse to every stratum $\tS_q$;

\smallskip
\item
every $\tS_q$ and every component of $\tE$ is invariant with respect to the action of the orbifold covering group (in every chart),
so there are induced strata $S_q := \tS_{q,\quot}$ and an induced divisor $E' := \tE_\quot$ in $X'$;

\smallskip
\item $S_q$ is the closure in $X'\backslash \bigcup_{q<r\leq p} S_r$ of the $\nc(q)$-locus of $X$.
\end{enumerate}

\smallskip
\item For each $q=1,\ldots p$,
there is a \emph{distinguished divisor} $\tD_q \subset \tE$ such that, if $D_q$ denotes  the induced \emph{distinguished divisor}
$\tD_{q,\quot} \subset X$', then:
\begin{enumerate}

\smallskip
\item the orbifold covering groups act freely on $\tX \backslash \sum_{q\leq p} \tD_q$ ; in particular, the projection of the latter 
to $X'$ is \'etale;

\smallskip
\item $X'$ has an orbifold group circulant singularity at every point of $S_q \cap D_q$
outside $\sum_{q<r\leq p} D_r$, and an $\nc(q)$ singularity at every point of $S_q$ outside $\sum_{q\leq r\leq p} D_r$.
\end{enumerate}
\end{enumerate}

\smallskip
Moreover, an action of a group $G$ on $X$ has a unique lifting to an action of $G$ on $X'$, with respect to which 
all strata $S_q$ and all components of $E'$ are invariant, and $\s$ is equivariant.
\end{theorem}

\begin{remark}\label{rem:orb1pintro} The distinguished divisors $\tD_q$, $q=p,p-1,\ldots,1$, arise in the recursive construction
as the exceptional divisors of weighted blowings-up of group circulant singularities of order $q$ (cf. Example \ref{ex:orbcpkintro}).
To compare with the low-dimensional results recalled in \S\ref{subsec:overview},
let us note the following.

\smallskip\noindent
(1) For any given $q$, the points of the distinguished divisor $D_q$ which lie outside $\sum_{q<r\leq p} D_r$ 
are the neighbouring singularities
(except for $\nc(q)$) of the orbifold group circulant singularities $S_q \cap D_q$.

\medskip\noindent
(2) The points of $D_q \bigcap \sum_{q<r\leq p} D_r$ are limits of these neighbouring singularities. Formulas for the limiting
singularities are given in Theorem \ref{thm:invtnc}. We note that these limiting singularities
are obtained in the inductive construction from the partial desingularization algorithm preserving normal crossings points
of order $< q$.

It follows that, for a given dimension $n$, there are finitely many distinct \'etale coordinate charts, invariant under corresponding 
actions of finite abelian
groups, needed to cover $\tX$, for any $X$. Therefore, a finite number of singular forms in addition to normal
crossings are needed to cover $\tX_\quot$, for any $X$. Moreover, one can give an explicit bound depending only on $n$ and $p$,
on the order of the finite abelian groups acting on the orbifold covering charts; cf. Example \ref{ex:orbgencirc}.
\end{remark}

\begin{remark}\label{rem:strength}
Theorem \ref{thm:orb1pintro} also can be strengthened in the ways indicated in Remarks \ref{rem:ordpintro}.
\end{remark}

Theorem \ref{thm:orb1pintro} will be deduced from Theorem \ref{thm:ordpintro} in Section \ref{sec:wtblup} by induction on $p$, in the following way.
Given a morphism $\s_p: X_p \to X$ which satisfies Theorem \ref{thm:ordpintro}, we construct a finite sequence of weighted blowings-up
of the locus of (product) group-circulant singularities (which are not already $\nc$) in the stratum $S_p$, 
as in Example \ref{ex:orbcpkintro}, after which $\tX_p$ has only normal
crossings, of order $< p$, in a deleted neighbourhood of $\tS_p$. 

Each weighted blowing-up has centre with support given by the intersection of $S_p$ with a component of the exceptional
divisor; at a given group-circulant singularity in $S_p$, the centre is a codimension one stratum in the sense
of Remark \ref{rem:groupcircintro}.

To continue, we then apply the induction hypothesis in the the complement of $\tS_p$; the centres of blowing-up will be isolated from $\tS_p$
(and therefore closed) because $\tX_p$ is already normal crossings of order $< p$ in a deleted neighbourhood of $\tS_p$.
Although the induction hypothesis is applied in the orbifold covering charts, we will show using \cite{Wlodar} that we nevertheless get
a composite of birational morphisms of the quotient varieties.

 \medskip
The final condition (5)(b) of Theorem \ref{thm:orb1pintro} involves orbifold (product) group-circulant singularities in
 the strata $S_q$ of $X'$, $q\leq p$, but we can recover the original group-circulant singularities provided by Theorem
 \ref{thm:ordpintro}, by simple additional blowings-up as in Remark \ref{rem:orbcpkintro} which, in particular, resolve
 the toric singularities of the quotients of the orbifold charts.
 
 \begin{addendum}\label{add:orb1pintro}
Beginning with the morphism $\s: X' \to X$ of Theorem \ref{thm:orb1pintro}, there is an additional sequence of (standard)
 blowings-up of the orbifold group singularities in the strata $S_q$, $q\leq p$, after which these strata satisfy the condition
 (5)(b) with only group circulant singularities.
 \end{addendum}
 
The argument will be given also in Section \ref{sec:wtblup}.

\section{Splitting theorem}\label{sec:splitthm}

In this section, we prove the Splitting Theorem \ref{thm:splitintro}. We use the notation of \S\ref{subsec:split}.

\begin{remark}\label{rem:localglobal}
\emph{Local versus global.} It is enough to prove the assertion of Theorem \ref{thm:splitintro} in some neighbourhood of any given
point $a_0\in S\cap E$. For the global statement, we can blow up the components $F_i := \{z=x=w_i = 0\},\, i=1,\ldots,r$,
of $S\cap E$ in the order determined by the given ordering of $E$, using a bound on the number of blowings-up of each $F_i$
given the local version. We could also use a bound on the powers $p$ given by the local statement, but this is 
unnecessary because $f$ splits over $\IC\llb w^{1/p},u,  x\rrb$, for some $p$, if and only if it splits with $p = k!$,
by \cite[Lemma 3.5]{BBR}. The divisor $E$ is essentially irrelevant to the local assertion, in the sense that we can simply
take $E$ to mean $\{w_1\cdots w_r = 0\}$.
\end{remark}

Our proof of Theorem \ref{thm:splitintro} uses a multivariate Newton-Puiseux theorem due to Soto and Vicente \cite{SV}, in the
same way it is used in the proof of \cite[Thm.\,1.13]{BBR}. We also use \cite[Lemma 3.1]{BBR} and the three lemmas below.
Lemmas \ref{lem:nonsplit} and \ref{lem:Yempty} in the algebraic case hold for any algebraically closed field $\IK$, but the proof
of Theorem \ref{thm:splitintro} requires an uncountable algebraically closed field $\IK$
of characteristic zero.

Given $f$ as in \eqref{eq:weierpoly},
let $Y\subset S\cap E$ denote the set of points $a\in S$ where $f$ does not split; i.e., where $f$ does not split over 
 $\IC\llb w-w_0, u-u_0,  x\rrb$, where $a=(w_0,u_0,0)$.
We will call $Y$ simply the \emph{nonsplitting locus} of $f$ in $S$.

\begin{lemma}\label{lem:nonsplit}
The nonsplitting locus $Y$ is a closed algebraic (or analytic) subset of $S\cap E$.
\end{lemma}

\begin{proof}
The assertion is an immediate consequence
of the fact that $Y$ is the complement in $S$ of the points $b\in S$ where there
are precisely $k$ points in the fibre of the normalization morphism for $X$.
\end{proof}

\begin{lemma}\label{lem:Yempty}
Suppose that, for every irreducible component $F_j = \{z=x=w_j =0\}$ of $S\cap E$, there is a point $a_j \in F_j\backslash Y$.
Then $Y = \emptyset$.
\end{lemma}

\begin{proof}
By Lemma \ref{lem:nonsplit}, $Y$ is a closed algebraic (or analytic) subset of $S\cap E$. The hypothesis therefore implies
that $Y$ has codimension at least $2$ in $S$. 

The proof of Lemma \ref{lem:Yempty} in the algebraic case involves an argument of the nature of Hartog's theorem---see
Remark \ref{rem:Yempty} following for the analytic case. Set $V := \{z=0\}$ and let $\tX \to X$ denote the normalization
of $X:= \{f(w,u,x,z) = 0\}$. Consider the branched covering $p: \tX \to X\to V$ of $V$, and let $B\subset V$ denote the
branch locus (the subset of $V$ over which $p$ is not \'etale). Then $B\cap S = Y$.

Over any algebraically closed field $\IK$, $B$ is a hypersurface $\{h(w,u,x)=0\}$ in $V$, by purity of the branch locus \cite{Nag1, Nag2}.
Suppose $Y\neq\emptyset$. Then $Y=B\cap S$ is a hypersurface $\{h(w,u,0)=0\}$ in $S$, in contradiction to $\codim Y \geq 2$.
\end{proof}

\begin{remark}\label{rem:Yempty}
To prove Lemma \ref{lem:Yempty} in the complex-analytic case,
we can assume there is an open polydisk $U$ with respect to the coordinates $(w,u,x)$ of $\{z=0\}$, centred at $a_0 =0$, such that,
for every $a\in (S\backslash Y)\cap U$, $f$ has $k$ local analytic roots $b_{a,i}(w,u,x)$, $i=1,\ldots,k$; say
\begin{equation*}
b_{a,i}(w,u,x) = \sum_{\al \in \IN^{k-1}} b_{a,i \al}(w,u)x^\al,
\end{equation*}
defined in a smaller polydisk $U_a\subset U$ centred at $a$, where the $b_{a,i \al}$ are analytic functions on $S\cap U_a$.
If $a,\,b\in (S\backslash Y)\cap U$ and $U_a\cap U_b\neq \emptyset$, then the roots at $a$ and $b$ (suitably ordered)
glue together on $U_a\cup U_b$.

Since $Y$ has codimension $\geq 2$ in $S$, $(S\backslash Y)\cap U$ is simply connected. By analytic continuation,
therefore, the coefficients
$b_{a,i \al}(w,u)$ of the local roots glue together to define analytic functions $b_{i \al}(w,u)$ on $(S\backslash Y)\cap U$. Now, by
Hartog's theorem, the $b_{a,i \al}$ extend to analytic functions on $S\cap U$, and the result follows.
 \end{remark}

\begin{lemma}\label{lem:generic}
Let $\IK$ denote an algebraically closed field of characteristic zero. For every $\la \in \overline{\IK(u)} = \overline{\IK(u_1,\ldots,u_q)}$,
there is a proper algebraic subset $W$ of $\IK^q$ such that, at every point $a=(a_1,\ldots,a_q) \in \IK^q\setminus W$, 
$\la$ is given by an element of the power series ring $\IK\llb u-a \rrb$.
\end{lemma}

\begin{proof}
Note that the tautological inclusion $\IK[u] \subset \IK\llb u\rrb$ identifies the algebraic closure $\overline{\IK(u)}$ as
the subfield of $\overline{\IK\llbr u\rrbr}$ consisting of elements of the latter that are algebraic over $\IK(u)$ (or over
$\IK[u]$, by clearing denominators). Given a point $u=a \in \IK^q$, there is a canonical isomorphism $u \mapsto u + (u-a)$
of $\IK(u)$ with $\IK(u-a)$.

Given $\la \in \overline{\IK(u)}$, there is a polynomial $P(u,t)$ in $t$ with coefficients in $\IK[u]$, such that $\la$ is a 
root of $P(u,t)=0$ and the discriminant $\De(u)$ of $P$ is not identically zero.

Let $a\in \IK^q\backslash W$, where $W$ is the zero-set of $\De$. If $m=\text{degree}\,P$, then $P(u,t) = P(a+(u-a), t)$
splits at $a$ with $m$ distinct roots $t=b_i(u-a) \in \IK\llb u-a\rrb \subset \overline{\IK\llbr u-a\rrbr},\, i=1,\ldots,m$, 
so that $\la = b_i$, for some $i$.
\end{proof}

\begin{proof}[Proof of Theorem \ref{thm:splitintro}]

\noindent
\emph{Case} $r=1$. Let $\IL$ denote the algebraic closure $\overline{\IC(u)}$ of the field of fractions of $\IC[u] = \IC[u_1,\ldots,u_q]$.
Let us write $(y_1,\ldots,y_k) := (x_1,\ldots,x_{k-1}, w) \allowbreak = (x,w)$.

Let $\SLplus(k,\IZ)$ denote the multiplicative subsemigroup of $\SL(k,\IZ)$ consisting of 
upper-triangular matrices
$$
A = \left(\begin{matrix}
1 & a_{12} & \cdots & a_{1k}\\
0 & 1         & \cdots & a_{2k}\\
\vdots  & \vdots  & \ddots & \vdots\\
0 & 0         & \cdots & 1
\end{matrix}\right)\,,
$$
where the $a_{ij}$ are nonnegative integers. Clearly, $\SLplus(k,\IZ)$ acts on monomials
$y^\al = y_1^{\al_1}\cdots y_k^{\al_k}$ by $y^\al \mapsto y^{\al A}$, $A \in \SLplus(k,\IZ)$, where
$$
\al A := (\al_1,\ldots,\al_k)\cdot\left(\begin{matrix}
1 & a_{12} & \cdots & a_{1k}\\
0 & 1         & \cdots & a_{2k}\\
\vdots  & \vdots  & \ddots & \vdots\\
0 & 0         & \cdots & 1
\end{matrix}\right)\,.
$$
Write $\psi_A(y^\al) := y^{\al A}$. Of course, $\psi_A$ extends to an operation on
$\IL\llb y \rrb = \IL\llb y_1,\ldots, y_k\rrb$, and to an operation on $\IL\llb y\rrb [z]$ (by the
preceding operation on coefficients), which we also denote $\psi_A$, in each case.

Since $\psi_A$ takes $y_k = w \mapsto w$ and (for each $i=1,\ldots,k-1$) takes $y_i \mapsto y_i$
times a monomial in $(y_{i+1},\ldots,y_{k-1}, w)$ (the monomial with exponents given by the $i$th row
of $A$), we see that $\psi_A$ also makes sense as an operation on $\overline{\IL(w)}\llb y_1,\ldots,y_{k-1}\rrb$,
or on $\overline{\IL\llbr w\rrbr}\llb y_1,\ldots,y_{k-1}\rrb$.

By the theorem of Soto and Vicente \cite{SV}, there exists a positive integer $p$ such that $f$ splits
in $\IL\llbr y_k^{1/p}\rrbr \cdots \llbr y_1^{1/p}\rrbr [z]$ and, moreover, there exists $A \in \SLplus(k,\IZ)$ such
that $\psi_A(f)$ splits in $\IL\llb y_1^{1/p},\ldots, y_k^{1/p}\rrb [z]$. Let $c_i \in \IL\llb y_1^{1/p},\ldots, y_{k-1}^{1/p}, w^{1/p}\rrb$,
$i=1,\ldots,k$, denote the roots of $\psi_A(f)$.

By \cite[Lemma 3.1]{BBR}, $f$ splits in $\overline{\IL(w)}\llb y_1,\ldots, y_{k-1}\rrb [z]$. Let $b_i \in \overline{\IL(w)}\llb y_1,\ldots, y_{k-1}\rrb
\subset \overline{\IL\llbr w\rrbr}\llb y_1,\ldots, y_{k-1}\rrb$, $i=1,\ldots,k$, denote the roots of $f$.

Since $\IL \llb y_1^{1/p},\ldots, y_{k-1}^{1/p}, w^{1/p}\rrb$ is a unique factorization domain, 
the set $\{c_i\}$ of roots of $\psi_A(f)$ coincides with the set $\{\psi_A(b_i)\}$;
i.e., each $c_i \in \IL \llb y_1,\ldots, y_{k-1}, w^{1/p}\rrb$.

Note that, given any monomial $y_1^{\al_1}\cdots y_{k-1}^{\al_{k-1}}$, $w=y_k$ appears in 
$\psi_A(y_1^{\al_1}\cdots y_{k-1}^{\al_{k-1}})$
to the power $\sum_{j=1}^{k-1} \al_j a_{jk}$; i.e., to a power at most $d\mu$, where $d$ is the degree $\al_1 +\cdots +\al_{k-1}$ and
$\mu = \max\{a_{jk}\}$.

It follows that blowing up $\mu$ times with centre $\{z=x_1=\cdots =x_{k-1} =w=0\}$ will clear all denominators in the roots $b_i$.
After these blowings-up, therefore, we can assume that $f(v^p,u,x,z)$ splits over $\overline{\IC(u)}\llb v,x\rrb$; i.e., splits with roots
\begin{equation*}
b_i(v,u,x) = \sum_{j\in \IN,\,\al\in\IN^{n-1}} b_{i,j\al}(u) v^j x^\al \in \overline{\IC(u)}\llb v,x\rrb,\quad i=1,\ldots,k.
\end{equation*}

By Lemma \ref{lem:generic}, for every $i,j,\al$, there is a closed algebraic subset $W_{i,j\al}$ of $\IC^q$ such that
$b_{i,j\al} \in \IC\llb u-u_0\rrb$, for every $u_0 \in \IC^q\backslash W_{i,j\al}$. Since 
a smooth variety over any uncountable algebraically closed field is not a countable union of proper subvarieties
(or, by the Baire category theorem, in the analytic case), there
exists $u_0 \in \bigcap_{i,j,\al} \left(\IC^q\backslash W_{i,j\al}\right)$, and it follows that $f$ splits over $\IC\llb v,u-u_0,x\rrb$.
The result for $r=1$ then follows from Lemma \ref{lem:Yempty}.

\emph{Case} $r>1$. We can treat  the $F_j = \{z=x=w_j = 0\}$, $j=1,\ldots,r$, one at a time, as in the case $r=1$ above, 
in order to conclude that, after finitely many blowings-up with centres of the form $F_j$, for each $j$, there is a positive integer $p$ such that $f(v^p,u,x,z) = f(v_1^p,\ldots,v_r^p.u,x,z)$ splits at a point (any point of a dense subset) of each $\{z=x=v_j = 0\}$.
The result follows again from Lemma \ref{lem:Yempty}.
\end{proof}

\section{Group-circulant normal form}\label{sec:circ}

We begin by showing how the Splitting Theorem \ref{thm:splitintro} leads to the introduction of a finite abelian group $\Ga$ and
the idea of a $\Ga$-circulant singularity or $\Ga$-circulant normal form (Definition \ref{def:groupcircintro}). The techniques used are
combinatorial, and involve the \emph{Pontryagin dual} of a finite abelian group \cite{Pont}. We prove that group-circulant normal
form induces standard product circulant normal form at every point of a codimension one stratum (Theorem \ref{thm:codim1}; see Remark
\ref{rem:groupcircintro}). Finally, we prove Theorem \ref{thm:ordpintro} on reduction to group-circulant normal form.

\medskip
Let $f(w,u,x,z)$ denote a function \eqref{eq:weierpoly} satisfying the conclusion of Theorem \ref{thm:splitintro}. Thus
$f(v_1^p,\ldots,v_r^{p},u,x,z)$
splits over $\IC\llb v,u,  x\rrb $, for some $p\in\IN$. Let $\cS_k$ denote the group of permutations of the roots of 
$f(v_1^p,\ldots,v_r^{p},u,x,z)$.

Assume that $f(w,u,x,z)$ is irreducible. Then
$(\IZ_{p})^r$ maps onto a subgroup of $\cS_k$ which acts transitively on the
roots. It follows (as in the proof of \cite[Lemma 3.5(2)]{BBR}) that 
\begin{equation}\label{eq:splitp}
f(v_1^{p_1},\ldots,v_r^{p_r},u,x,z)
\end{equation} splits, 
where, for each $i=1,\ldots,r$,\, $p_i \leq k$ and the group $\IZ^{(i)}_{p_i} := \{0\}^{i-1} \times\IZ_{p_i}\times\{0\}^{r-i}$ maps
onto a cyclic subgroup  $\IZ_{p_i}$ of $\cS_k$. Let $G$ denote the abelian group $\IZ_{p_1} \times \cdots \times \IZ_{p_r}$ 
of order $p_1\cdots p_r$.

We can write
\begin{equation}\label{eq:splitp1}
f(v_1^{p_1},\ldots,v_r^{p_r},u,x,z) = \prod_{j=1}^k (z + b_j(v,u,x)),
\end{equation}
where each $b_j \in \IC\llb v,u,x\rrb$. Since $f$ is irreducible, taking $b := b_{j_0}$, for any $j=j_0$, 
all roots $-b_j$ of \eqref{eq:splitp1} 
are of the form 
\begin{equation}\label{eq:roots.1}
-b(\ep_{p_1}^{j_1} v_1,\ldots,\ep_{p_r}^{ j_r}v_r, u,x),\quad j_i =0,\ldots, p_i-1,\ \ i=1,\ldots, r,
\end{equation}
where $\ep_p$ denotes $e^{2\pi i/p}$, so that $\ep_{p_i} = \ep_k^{k/p_i}$.

 The group $G$ acts
on $\IC\llb v,u,x,z\rrb$ as in \eqref{eq:roots.1}; i.e., $j=(j_1,\ldots,j_r) \in G$ takes $\vp(v,u,x,z) \in \IC\llb v,u,x,z\rrb$
to $\vp(\ep_{p_1}^{j_1} v_1,\ldots,\ep_{p_r}^{ j_r}v_r, u,x,z)$. This action induces a homomorphism of $G$ onto an abelian
subgroup of the group
$\cS_k$ of permutations of the roots of \eqref{eq:splitp1}, which acts transitively on the roots.

For each $i$, $p_i$ divides $k$, since the translates of $\IZ_{p_i} \subset S_k$ by the elements of $G$ provide 
a partition of the set of roots of \eqref{eq:splitp1} into subsets of $p_i$ elements. 

Since $G$ is abelian, the stabilizer $H$ of any root by the action of $G$
is independent of the root; $H$ is a subgroup of $G$ of order $|H| = p_1\cdots p_r/k$.

For each $j=(j_1,\ldots,j_r) \in G$, set
\begin{equation}\label{eq:factor}
Y_j := z + b(\ep_{p_1}^{j_1} v_1,\ldots,\ep_{p_r}^{ j_r}v_r, u,x).
\end{equation}
Generalizing \cite[Sect.\,2]{BBR}, write
\begin{equation}\label{eq:circinv.1}
X_\ell := \frac{1}{p_1\cdots p_r} \sum_{j_1=0}^{p_1-1} \cdots \sum_{j_r=0}^{p_r-1} \ep_{p_1}^{\ell_1(p_1 -j_1)}\cdots
 \ep_{p_r}^{\ell_r(p_r -j_r)} Y_j,\quad \ell \in G.
 \end{equation}
 
 In \S\ref{subsec:circcomb} following, $Y_j$, $j\in G$, denote independent variables. The expressions
 \eqref{eq:factor} for the factors of
$f(v_1^{p_1},\ldots,v_r^{p_r},u,x,z)$ play no part in \S\ref{subsec:circcomb}, but will reappear
 in an important way in the following subsections.

 \subsection{Circulant combinatorics}\label{subsec:circcomb}
 Let $Y_j$, $j=(j_1,\ldots,j_r) \in G$, denote indeterminates, and define $X_\ell$, $\ell = (\ell_1,\ldots, \ell_r) \in G$,
 by \eqref{eq:circinv.1}. Thus \eqref{eq:circinv.1}
 defines a linear transformation $W\to V$, where $W$ and $V$ are copies of $\IC^{|G|} = \IC^{p_1\cdots p_r}$. This transformation
 is invertible, and its inverse is given by
 \begin{equation}\label{eq:circ.1}
 Y_j = \sum_{\ell_1=0}^{p_1-1} \cdots \sum_{\ell_r=0}^{p_r-1} \ep_{p_1}^{j_1\ell_1}\cdots
 \ep_{p_r}^{j_r\ell_r}X_\ell,\quad j \in G
 \end{equation}
(cf. \cite[(2.2)]{BBR}).

Let $k$ denote the least common multiple (or any common multiple) of $p_1,\ldots,p_r$. 
Note that $\ep_{p_i}^{\ell_i(p_i -j_i)} = \ep_k^{-(k/p_i)\ell_i j_i}$.

\begin{remark}\label{rem:groups}
$\IZ_{p_i}$ \emph{as a subroup of} $\IZ_k$. The inclusion homomorphism $\IZ_{p_i} \hookrightarrow \IZ_k$,
as additive groups of integers mod $p_i$ or $k$, takes $1 \in \IZ_{p_i}$ to $k/p_i \in \IZ_k$. Regarding $\IZ_{p_i},\, \IZ_k$
as multiplicative groups of roots of unity, this means simply writing $\ep_{p_i}$ as $\ep_k^{k/p_i}$.
\end{remark}

Write $\ep := \ep_k$ for brevity. We can rewrite the equations \eqref{eq:circ.1}, \eqref{eq:circinv.1} as
\begin{equation}\label{eq:circ.2}
 Y_j =  \sum_{\ell \in G} \ep^{\langle j, \ell\rangle} X_\ell,\quad j \in G,
 \end{equation}
 with inverse 
 \begin{equation}\label{eq:circinv.2}
X_\ell = \frac{1}{p_1\cdots p_r} \sum_{j\in G} \ep^{- \langle \ell, j\rangle} Y_j, \quad \ell \in G,
\end{equation}
where $\langle j, \ell\rangle$ denotes the ``weighted scalar product'' 
\begin{equation*}\label{eq:angle}
\langle j, \ell\rangle := \sum_{i=1}^r \frac{k}{p_i}\, j_i \ell_i\,,\quad
j=(j_1,\ldots j_r),\, \ell = (\ell_1,\ldots, \ell_r) \in G.
\end{equation*}

If $Y_j$ is given by \eqref{eq:factor}, then $Y_j = Y_{j+h}$ if and only if $h\in H$. In general, let $H$ denote any subgroup
of $G$, and let $W_H$ denote the linear subspace of $W$ defined by the equations
\begin{equation*}
Y_j = Y_{j+h}, \quad j \in G,\ h \in H.
\end{equation*}
Then $\dim W_H = |G/H|$.

Let $t:= |G/H|$ and choose representatives $j^0,\ldots j^{t-1} \in G$ of the distinct elements of $G/H$. Then, for every $\ell\in G$,
\begin{equation}\label{eq:circinv.3}
X_\ell = \frac{1}{p_1\cdots p_r} \sum_{i=0}^{t-1}\sum_{h\in H} \ep^{- \langle \ell, j^i + h\rangle} Y_{j^i + h}\,.
\end{equation}
If $Y = (Y_j)_{j\in G} \in W_H$, then
\begin{equation}\label{eq:circinv.4}
X_\ell = \frac{\xi_\ell}{p_1\cdots p_r} \sum_{i=0}^{t-1} \ep^{- \langle \ell, j^i \rangle} Y_{j^i},
\end{equation}
where
\begin{equation}\label{eq:xi}
\xi_\ell := \sum_{h\in H} \ep^{- \langle \ell, h \rangle} .
\end{equation}

Let
\begin{equation*}
K := \left\{\ell \in G: \ep^{\langle \ell, h\rangle} = 1,\ \text{for all } h\in H\right\};
\end{equation*}
i.e., $K = \{\ell\in G: \langle \ell, h\rangle \equiv 0 \mod k,\ \text{for all } h\in H\}$, so it makes sense to call $K$ the
``orthogonal complement'' $K = H^\perp$ of $H$ in $G$. It is easy to see that $H^\perp$ is a subgroup of $G$.

\begin{example}\label{ex:perp}
Let $G = \IZ_k \times \cdots \times \IZ_k$ ($r$ times), and let $H$ denote the subgroup $\IZ_{p_1} \times \cdots \times \IZ_{p_r}$ of
$G$. Then $H^\perp = \IZ_{k/p_1} \times \cdots \times \IZ_{k/p_r}$.

In particular, if $G=\IZ_4$ and $H=\IZ_2 \subset \IZ_4$, then $H^\perp = H$ and the canonical homomorphism
$H^\perp \hookrightarrow G \to G/H$ is not an isomorphism.
\end{example}

\begin{lemma}\label{lem:isom}
The orthogonal complement $H^\perp$ is (non-canonically) isomorphic to $G/H$.
\end{lemma}

 \begin{proof}
We use the \emph{(Pontryagin) dual} $G^*$ of $G$ \cite{Pont}. The dual $G^*$ of a finite abelian group $G$ is defined as
\begin{equation}\label{eq:dual}
G^* := \Hom(G, \IQ/\IZ).
\end{equation}
(The \emph{character group} $\Hom(G, \IC^*) \cong \Hom(G, \IQ/\IZ)$, for $G$ finite abelian.)

Suppose $G$ is a finite abelian group. First we observe that $G^*$ is non-canonically isomorphic to $G$. We can assume
that $G = \IZ_{p_1}\times\cdots\times \IZ_{p_r}$. Then 
\begin{equation}\label{eq:oplus}
\Hom(G, \IQ/\IZ) = \bigoplus_{i=1}^r \Hom(\IZ_{p_i},\IQ/\IZ)
\end{equation}
(if $f \in \Hom(G, \IQ/\IZ)$, we write $f = (f_1,\ldots,f_r)$ with respect to \eqref{eq:oplus}).
So it is enough to show that $(\IZ_p)^* \cong \IZ_p$.
Any $f \in \Hom(\IZ_{p},\IQ/\IZ)$ takes $1$ to $f(1) \in \IQ/\IZ$, and $f(1)$ has a representative $a/b \in \IQ$ with $a<b$,
in lowest terms. Since $pf(1) = 0$, we have $pa/b \in \IZ$; hence $b|p$ and
$$
\frac{a}{b} = \frac{a(p/b)}{p}.
$$
In other words, $f(1)$ has $p$ possible values $0, 1/p, \ldots, (p-1)/p$.

Given a scalar product $\langle \cdot, \cdot \rangle$ on $G$, there is a homomorphism
\begin{equation}\label{eq:isom}
\begin{aligned}
G &\to G^*\\
g &\mapsto \langle g, \cdot\rangle.
\end{aligned}
\end{equation}
Conversely, any element $f\in G^*$ can be realized in this way, taking $g = (f_1(1),\ldots,\allowbreak f_r(1))$,
so that \eqref{eq:isom} is an isomorphism.

Now, given a subgroup $H$ of $G$, we define the orthogonal complement $H^\mathrm{O}$ of $H$ in $G^*$ as
\begin{equation*}
H^\mathrm{O} := \{f \in G^*: f(h) =0,\ h\in H\}.
\end{equation*} 
The isomorphism \eqref{eq:isom} takes $H^\perp$ (with respect to the scalar product) to $H^\mathrm{O}$. The result follows since
$$
(G/H)^* = \Hom(G/H,\IQ/\IZ) = \{f \in \Hom(G, \IQ/\IZ): H \subset \ker f\} = H^\mathrm{O}.
$$
\end{proof}

\begin{lemma}\label{lem:xi}
Define $\xi_\ell$ as in \eqref{eq:xi}. Then
\begin{equation*} 
\xi_\ell \,=\, 
\begin{cases}
|H|, & \ell\in H^\perp,\\
0, & \ell\notin H^\perp.
\end{cases}
\end{equation*}
\end{lemma}

\begin{proof}
The assertion is clear in the case $\ell \in H^\perp$. On the other hand,
if $\ell\notin H^\perp$, then there exists $h_{\ell} \in H$ such that $\ep^{\langle \ell, h_\ell \rangle} \neq 1$. But
\begin{equation*}
\ep^{\langle \ell, h_\ell \rangle} \xi_\ell = \sum_{h\in H} \ep^{\langle \ell, h_\ell \rangle} \ep^{\langle \ell, h\rangle} 
= \sum_{h\in H} \ep^{\langle \ell, h_\ell + h\rangle} = \xi_\ell,
\end{equation*}
so that $\xi_\ell = 0$.
\end{proof}

\begin{lemma}\label{lem:subspaces}
$Y_j = Y_{j+h}$, for all $j\in G$ and $h\in H$, if and only if $X_\ell = 0$, for all $\ell\notin H^\perp$.
\end{lemma}

\begin{proof}
If $Y_j = Y_{j+h}$, for all $j\in G$ and $h\in H$, then $X_\ell = 0$, for all $\ell\notin H^\perp$, by \eqref{eq:circinv.4} and
Lemma \ref{lem:xi}.
Conversely, suppose $X_\ell = 0$, for all $\ell\notin H^\perp$. If $j\in G$ and $h\in H$, then
\begin{equation*}
Y_{j+h} = \sum_{\ell\in H^\perp} \ep^{\langle j+h, \ell\rangle} X_\ell 
=  \sum_{\ell\in H^\perp} \ep^{\langle j, \ell\rangle} \ep^{\langle h, \ell\rangle} X_\ell = Y_j ,
\end{equation*}
since $\ep^{\langle h, \ell\rangle} = 1$ if $h\in H$ and $\ell \in H^\perp$.
\end{proof}

\begin{corollary}\label{cor:subspaces}
The linear transformation \eqref{eq:circinv.1} (or \eqref{eq:circinv.2}) induces an isomorphism $W_H \to V_H$, where
\begin{equation*}
V_H := \{(X_\ell)_{\ell\in G} \in V: X_\ell = 0,\ \text{for all }\, \ell \notin H^\perp\};
\end{equation*}
in particular, $|H^\perp| = |G/H|$ (which also follows from Lemma \ref{lem:isom}).
\end{corollary}

\medskip
We can interpret the preceding combinatorial remarks in terms of group-circulant
matrices. Given $j = (j_1,\ldots,j_r) \in G$, let $\oj$ denote the class of $j$ in $G/H$,
and let $Y_{\oj} := Y_j|_{W_H}$. Then the inverse of the isomorphism $W_H \to V_H$ of Corollary \ref{cor:subspaces}
is given by the linear transformation
\begin{equation*}
Y_{\oj} = \sum_{\ell\in H^\perp} \ep^{\langle j,\ell\rangle} X_\ell, \quad \oj \in G/H.
\end{equation*}

Let $\ell_0, \ell_1,\ldots, \ell_{t-1}$ denote any enumeration of the indices $\ell \in H^\perp$, with $\ell_0 = 0$. 
For the additive group $\Ga := G/H$, the circulant matrix \eqref{eq:Gcirc} takes the form
\medskip
\begin{equation}\label{eq:circmat}
C_{G/H} (X_{\ell_0}, X_{\ell_1},\ldots, X_{\ell_{t-1}}) \,= 
\begin{pmatrix}
X_{\ell_0} & X_{\ell_1} & \cdots & X_{\ell_{t-1}} \\[.2em]
X_{\ell_0 - \ell_1} & X_{\ell_1 - \ell_1} & \cdots & X_{\ell_{t-1} - \ell_1} \\[.2em]
\vdots & \vdots & \ddots & \vdots \\[.2em]
X_{\ell_0 - \ell_{t-1}} & X_{\ell_1 - \ell_{t-1}}  & \cdots & X_{\ell_{t-1} - \ell_{t-1}}
\end{pmatrix}.
\end{equation}

\begin{remark}\label{rem:circdet}
The $\Ga$-circulant matrix $C_{\Ga}(X_{\ell_0}, X_{\ell_1},\ldots, X_{\ell_{t-1}})$
depends on $\Ga$, as well as on the ordering $\ell_0, \ell_1,\ldots, \ell_{t-1}$.
This can be compared with the standard circulant matrix \cite[(2.1)]{BBR}, which is $C_\Ga$ in the case $\Ga = \IZ_k$, with
the ordering $\ell = 0,1,\ldots,k-1$.

On the other hand, let
\begin{equation*}\label{eq:circdet}
\De_{\Ga} := \det C_{\Ga}\,;
\end{equation*}
then both $\De_{\Ga}$ and the eigenvalues of $C_{\Ga}$ are independent of the 
ordering $\ell_0, \ell_1,\ldots, \allowbreak \ell_{t-1}$; see \S\ref{subsec:Gcirc}.

In particular, 
$\Ga \cong \IZ_t$ if and only if we can take $\ell_i = i$, $0 \leq i \leq t-1$. Thus $\De_{\Ga}$ coincides with
the standard circulant determinant $\De_t(X_0,\ldots X_{t-1})$ if and only if $\Ga$ is a cyclic group of order $t$.

The matrix in \eqref{eq:circmat} depends on the group structure. If we introduce a change of variables
$Y_i := X_{\ell_i}$, $i=0,\ldots,t-1$, then we can write $\De_\Ga(Y_0,\ldots, Y_{t-1})$, following standard 
practise, to mean the composite
of the determinant $\De_\Ga(X_{\ell_0}, X_{\ell_1},\ldots, X_{\ell_{t-1}})$ with the transformation $X_{\ell_i} = Y_i$.
Note, however, that this notation hides the group structure; in particular, $\De_\Ga(Y_0,\ldots, Y_{t-1})$ is not the determinant 
of a matrix whose successive rows are obtained
from the first row $(Y_0,\ldots, Y_{t-1})$ by a trivial (e.g., cyclic) permutation, unless $\Ga$ is a cyclic group.
See Example \ref{ex:gencirc}.
\end{remark}

It is easy to check that the matrix in \eqref{eq:circmat} has eigenvectors
\begin{equation*}
\Psi_{\oj} = (\ep^{\langle j, \ell_o\rangle},\, \ep^{\langle j, \ell_1\rangle}, \ldots,\, \ep^{\langle j, \ell_{t-1}\rangle}), \quad \oj \in G/H,
\end{equation*}
with corresponding eigenvalues
\begin{equation}\label{eq:eigengen}
Y_{\oj} = \sum_{i=0}^{t-1} \ep^{\langle j,\ell_i\rangle} X_{\ell_i}\,;
\end{equation}
in particular,
\begin{equation*}
\det C_{G/H} (X_{\ell_0}, X_{\ell_1},\ldots, X_{\ell_{t-1}}) = \prod_{\oj \in G/H} Y_{\oj}.
\end{equation*}

\subsection{Group circulant normal form}\label{subsec:groupcirc}
Now consider $f(w,u,x,z)$ irreducible, as at the beginning of Section \ref{sec:circ}. 
We can assume that $a_1(w,u,x) = 0$; i.e., $\sum_{j=1}^k b_j = 0$. We follow the notation at the beginning of the section,
so that
\begin{equation*}
f(w,u,x,z) = \prod_{\oj\in G/H} Y_{\oj},
\end{equation*}
where
\begin{equation*}
Y_{\oj} = Y_j := z + b(\ep_{p_1}^{j_1}w_1^{1/p_1}, \ldots, \ep_{p_r}^{j_r}w_r^{1/p_r}, u, x),
\end{equation*}
for any representative $j = (j_1,\ldots, j_r)$ of $\oj$ in $G = \IZ_{p_1}\times \cdots \times \IZ_{p_r}$. 

Recall that, for each $i=1,\ldots,r$, $\IZ^{(i)}_{p_i} := \{0\}^{i-1} \times\IZ_{p_i}\times\{0\}^{r-i}$ maps
onto a cyclic subgroup  $\IZ_{p_i}$ of the group $S_k$ of permutations of the factors $Y_{\oj}$, and the translates
of $\IZ_{p_i} \subset S_k$ by the elements of $G$ provide a partition of $\{Y_{\oj}\}$ into $k/p_i$ subsets of $p_i$ elements.

In particular, given $i=1,\ldots,r$. at any point of the \emph{codimension one stratum}
\begin{equation}\label{eq:codim1stratum}
T_i := \{z=x=w_i = 0,\, w_h\neq 0,\, h\neq i\},
\end{equation}
(i.e., codimension one in $S := \{z=x=0\}$), $f(w,u,x,z)$ formally has $k/p_i$ irreducible factors of degree $p_i$ in $z$,
corresponding to the preceding partition.

As before,  $H$ denotes the stabilizer of any root of \eqref{eq:splitp1} 
by the action of $G = \IZ_{p_1}\times \cdots \times \IZ_{p_r}$. Thus $|G/H| = k$.

We recall that $f$ is said to have a $G/H$-\emph{circulant singularity} at a point $a$ if, in suitable coordinates
$(w,u,x,z)$ in an \'etale or analytic neighbourhood of $a=0$,
\begin{equation}\label{eq:normform}
f(w,u,x,z) = \De_{G/H}(X_{\ell_0}, X_{\ell_1},\ldots, X_{\ell_{k-1}}),
\end{equation}
where the conditions (1) and (2)
of Definition \ref{def:groupcircintro} hold with $G/H = \{0 = \ell_0, \ell_1, \ldots,\allowbreak \ell_{k-1}\}$, and
\begin{align*}\label{eq:normform1}
X_{\ell_0} &= z,\nonumber\\
X_{\ell_j} &= w^{\ga_j} x_j,\quad j=1,\ldots, k-1;
\end{align*}
in this case, moreover, \eqref{eq:normform} is called $G/H$-\emph{circulant normal form}.
We also consider \emph{product group circulant normal form} in the case that $f$ is
not irreducible, as in Definition \ref{def:groupcircintro}.

The exponents 
$\ga_j$ in Definition \ref{def:groupcircintro} are not necessarily optimal, but
they are optimal on every codimension one stratum $T_i$.
The following theorem shows, moreover,
at any point $a\in T_i$, \eqref{eq:normform} can be rewritten in the product circulant normal form
of Example \ref{ex:standardcirc} at $a$. 

\begin{theorem}\label{thm:codim1}
\emph{Codimension one strata.} If $f(w,u,x,z)$ is irreducible and has group circulant normal form 
as above, then, at any point $b$ of a codimension one stratum $T_i$, we can find coordinates in an
\'etale or local analytic neighbourhood of $b$
in which \eqref{eq:normform} can be rewritten in the standard (product) circulant normal form of Example \ref{ex:standardcirc};
in particular, for each $j=1,\ldots, k-1$, $\{\ga_{ji}: i=1,\ldots,r\}$
comprises $k/p_i$ copies of $h/p_i$, for each $h = 0,\ldots, p_i-1$ (where $\ga_{0i} := 0$, $i=1,\ldots,r$).
\end{theorem}

In order to prove Theorem \ref{thm:codim1}, we will need to further develop the combinatorial techniques of \S\ref{subsec:circcomb}; this will be done in \S\ref{subsec:codim1combin} following. The proof of
Theorem \ref{thm:codim1} will be completed in \S\ref{subsec:codim1}. 

\begin{example}\label{ex:nonirred}
Consider the complex variety $X$ defined by
\begin{equation}\label{eq:nonirred3}
f(w,x,z) := \De_{4}\left(z,\, w_1^{2/4}w_2^{1/4} x_1,\, w_2^{2/4} x_2,\, w_1^{2/4}w_2^{3/4} x_{3}\right).
\end{equation}
Then $X$ is irreducible, $G = \IZ_2\times\IZ_4$ and $G/H \cong \IZ_4$. The normalization of $X$ is not smooth---it
is the product of $\IC^3$ with a quadratic cone---in contrast to $\cp(k)$.

At any point of the codimension one stratum $T_1 = \{z=x=w_1=0,\,w_2\neq 0\}$, \eqref{eq:nonirred3} can be written
\begin{equation*}\label{eq:nonirred6}
\De_{4}\left(z,\, w_1^{2/4} y_1,\, y_2,\, w_1^{2/4} y_{3}\right),
\end{equation*}
after a change of variables, and the latter factors as 
\begin{equation*}
\left((z+y_2)^2 - w_1 (y_1 +y_3)^2\right)\left((z-y_2)^2 - w_1 (y_1 -y_3)^2\right).
\end{equation*}
This is $\cp(2) \times \cp(2)$; i.e.,
$$
(\eta_1^2 -w_1\xi_1^2)(\eta_2^2 -w_1\xi_2^2),
$$
after a change of variables.
\end{example}

In the remainder of this subsection,
we prove three elementary lemmas which are not formally used in the proof of Theorem \ref{thm:codim1}, 
but which provide important motivation; in particular, Theorem \ref{thm:codim1} is a very
general version of Lemma \ref{lem:converse} below, and the proof includes a general parallel
of the identity in Lemma \ref{lem:prodk}
(Lemma \ref{lem:circnormform}
is also used implicitly in \cite{BBR}; see \cite[Remark 4.3]{BBR}.)

\begin{lemma}\label{lem:circnormform}
Suppose that $r=1$, $f(w,x,z)$ is irreducible, and 
\begin{align*}
f(w,x,z) &= \De_k\left(z, w^{h_1/k}x_1, w^{h_2/k}x_2\ldots, w^{h_{k-1}/k} x_{k-1}\right)\\
               &= \prod_{\ell = 0}^{k-1}\, (z + \ep^\ell w^{h_1/k} x_1 + \cdots + \ep^{(k-1)\ell}w^{h_{k-1}/k}x_{k-1}),
\end{align*}
where $\ep = e^{2\pi i/k}$ and $(h_1,\ldots,h_{k-1})$ is a permutation of $(1,\ldots,k-1)$ (see \cite[Remark 4.3]{BBR}). Then, after the
permutation of variables $x_j = y_{h_j}$, $j=1,\ldots,k-1$, $f$ can be written in circulant normal form
\begin{multline*}
\De_k\left(z, w^{1/k}y_1, w^{2/k}y_2\ldots, w^{(k-1)/k} y_{k-1}\right) \\= \prod_{\ell = 0}^{k-1}\, (z + \ep^\ell w^{1/k} y_1 + 
                                     \cdots + \ep^{(k-1)\ell}w^{(k-1)/k}y_{k-1}).
\end{multline*}
\end{lemma}

\begin{proof} After the permutation of variables $x_j = y_{h_j}$, $j=1,\ldots,k-1$, 
we can rewrite $f$ as
\begin{equation}\label{eq:permcirc}
\prod_{\ell = 0}^{k-1}\, (z + \ep^{j_1 \ell} w^{1/k} y_1 + \cdots + \ep^{j_{k-1}\ell}w^{(k-1)/k}y_{k-1}),
\end{equation}
where $(j_1,\ldots,j_{k-1})$ is a permutation of $(1,\ldots,k-1)$.

The factor of \eqref{eq:permcirc} with $\ell=0$ is $z + w^{1/k} y_1 + \cdots + w^{(k-1)/k}y_{k-1}$.
Note that the action of $\ep \in \IZ_k$ on the roots of $f$ takes 
$$
\ep^\ell w^{1/k} y_1 + \cdots + \ep^{(k-1)\ell}w^{(k-1)/k}y_{k-1},
$$
for any $\ell$, into the same expression with $\ell +1$. Since the action is transitive on the roots,
$$
f(w,y,z) = \prod_{\ell = 0}^{k-1}\, (z + \ep^\ell w^{1/k} y_1 + \cdots + \ep^{(k-1)\ell}w^{(k-1)/k}y_{k-1}),
$$
as required. (Note that the factors here are not in the same order as in \eqref{eq:permcirc}.)
\end{proof}

\begin{lemma}\label{lem:converse}
Suppose that
\begin{align}
f(w,x) &= \De_k\left(w^{h_0/k}x_0, w^{h_1/k}x_1, w^{h_2/k}x_2\ldots, w^{h_{k-1}/k} x_{k-1}\right)\notag\\
               &= \prod_{\ell = 0}^{k-1}\, (w^{h_0/k}x_0 + \ep^\ell w^{h_1/k} x_1 + \cdots + \ep^{(k-1)\ell}w^{h_{k-1}/k}x_{k-1})\label{eq:fact}
\end{align}
is a polynomial in $(w,x) = (w,x_0,x_1,\ldots,x_{k-1})$, where each $h_j \in \{0,\ldots,k-1\}$.
Then $f$ is irreducible if and only if $(h_0,\ldots,h_{k-1})$ is a
permutation of $\{0,\ldots,k-1\}$.
\end{lemma}

\begin{proof}
The generator $\ep$ of $\IZ_k$ acts on the factors of \eqref{eq:fact} by
\begin{equation*}
\ep^{j\ell} w^{h_j/k} x_j \mapsto \ep^{j\ell + h_j} w^{h_j/k} x_j, \quad j=0,\ldots, k-1.
\end{equation*}
Since $f(w,x)$ is invariant, the action permutes the factors; i.e.,
\begin{equation*}
j\ell + h_j = j \la_\ell \mod k, \quad j=0,\ldots,k-1,
\end{equation*}
where $(\la_0,\ldots,\la_{k-1})$ is a permutation of $(0,\ldots,k-1)$. Thus, $h_j =\mu j \mod k$, $j=0,\ldots,k-1$,
where we can take $\mu = \la_\ell - \ell$, for any given $\ell$. 

Then $f(w,x)$ is  irreducible if and only if $\mu$ and $k$ are relatively prime;
i.e., $(0, \mu,\ldots, (k-1)\mu)$ (each term$\mod k$) is a permutation of $(0,\ldots,k-1)$. The result follows from
(a slightly more general statement of) Lemma \ref{lem:circnormform}.
\end{proof}

\begin{lemma}\label{lem:prodk}
Given positive integers $k\geq 2$ and $r$,
\begin{equation}\label{eq:prodk}
\begin{aligned}
\prod_{i=0}^{r-1} \De_k (x_{i0}, &w^{1/k}x_{i1},\ldots, w^{(k-1)/k}x_{i,k-1})\,\\
&=\, \De_{rk} (x_0, w^{1/k}x_1,\ldots, w^{(k-1)/k}x_{k-1}, x_k, w^{1/k}x_{k+1},\ldots, w^{(k-1)/k}x_{2k-1},\\
&\hspace{5.2cm} \ldots, x_{(r-1)k},\ldots,w^{(k-1)/k}x_{2k-1}),
\end{aligned}
\end{equation}
where $(x_{ij})_{i=0,\ldots,r-1,\,j=0.\ldots,k-1}$ and $(x_{i})_{i=1,\ldots,rk}$ are related by the following identities:
\begin{equation}\label{eq:prodkid}
x_{ij}  = \ep_{rk}^{ij} ( x_j + \ep_r^i x_{k+j} + \cdots + \ep_r^{(r-1)i} x_{(r-1)k+j}),\quad
                   i=0,\ldots,r-1,
\end{equation}
for each $j=0.\ldots,k-1$ (note that \eqref{eq:prodkid} defines an invertible linear transformation, for each fixed $j$).
\end{lemma}

\begin{proof}
In the left-hand side of \eqref{eq:prodk}, each factor
\begin{equation}\label{eq:prodk1}
\begin{aligned}
\De_k (x_{i0}, w^{1/k}x_{i1},&\ldots, w^{(k-1)/k}x_{i,k-1})\\
&= \prod_{\ell=0}^{k-1} (x_{i0} + \ep_{k}^\ell w^{1/k}x_{i1} + \cdots + 
\ep_{k}^{(k-1)\ell}w^{(k-1)/k}x_{i,k-1}.
\end{aligned}
\end{equation}
The right-hand side of \eqref{eq:prodk} is given by
\begin{equation*}
\begin{aligned}
&\prod_{m=0}^{rk-1}(x_0 + \ep_{rk}^m w^{1/k}x_1 + \cdots + \ep_{rk}^{(k-1)m}w^{(k-1)/k}x_{k-1} +\ep_{rk}^{km}x_k +  \ep_{rk}^{(k+1)m}w^{1/k}x_{k+1}\\ 
&\hspace{1.5cm} + \cdots + \ep_{rk}^{(2k-1)m}w^{(k-1)/k}x_{2k-1} + \cdots + \ep_{rk}^{(rk-1)m}w^{(k-1)/k}x_{rk-1})\\
= &\prod_{m=0}^{rk-1} \left((x_0 + \ep_r^m x_k + \cdots + \ep_r^{(r-1)m}x_{(r-1)k})\right.\\
&\hspace{1.5cm} + w^{1/k} \ep_{rk}^m (x_1 + \ep_r^m x_{k+1} + \cdots + \ep_r^{(r-1)m}x_{(r-1)k+1})\\
&\hspace{1.5cm} \left. + \cdots + w^{(k-1)/k} \ep_{rk}^{(k-1)m} (x_{k-1} + \ep_r^m x_{2k-1} + \cdots + \ep_r^{(r-1)m}x_{rk-1}\right).
\end{aligned}
\end{equation*}
The indices $m$ can be written $m=\ell r + i$, where $\ell =0.\ldots,k-1,\, i = 0,\ldots,r-1$, so that 
$\ep_{rk}^{jm} = \ep_{rk}^{j\ell r +ij} = \ep_k^{j\ell} \ep_{rk}^{ij}$, $j=0,\ldots,k-1$, 
and the result follows.
\end{proof}

\subsection{Codimension one combinatorics}\label{subsec:codim1combin}
We use the notation of \S\ref{subsec:circcomb} and of the introduction to Section \ref{sec:circ}; 
in particular, $G = \IZ_{p_1}\times\cdots\times \IZ_{p_r}$, $H$
is the stabilizer of a root of \eqref{eq:splitp1}, $Y_j$, $j=(j_1,\ldots,j_r) \in G$, is given by \eqref{eq:factor}
and $X_\ell$, $\ell\in G$, by \eqref{eq:circinv.1}.

If $j \in G$, then $Y_j$ depends only on the class $\oj$ of $j$ in $G/H$, so we write
$Y_{\oj} := Y_j$. Then we have
\begin{equation}\label{eq:circeqn}
X_\ell = \frac{1}{k}\sum_{\oj\in G/H} \ep^{-\langle \ell,j\rangle} Y_{\oj},\quad \text{ for all }\, \ell \in H^\perp,
\end{equation}
or 
\begin{equation}\label{eq:circeqninv}
Y_{\oj} = \sum_{\ell \in H^\perp} \ep^{\langle j, \ell\rangle} X_\ell, \quad \text{ for all }\, \oj\in G/H,
\end{equation}
where $\ep = \ep_k := e^{2\pi i/k}$. Note that, if $j$ is any representative of $\oj$ in $G$, then 
$\ep^{-\langle \ell,j\rangle} = \ep_{p_1}^{\ell_1(p_1-j_1)}\cdots \ep_{p_r}^{\ell_1(p_r-j_r)}$,
where $j = (j_1,\ldots,j_r)$ and $\ell = (\ell_1,\ldots,\ell_r)$, depends only $\oj$.
Note also that $X_0 = z + a_1(w,u,x)/k$, in the notation of \eqref{eq:weierpoly}.

Recall that $\IZ_{p_i}^{(i)} := \{0\}^{i-1} \times \IZ_{p_i} \times \{0\}^{r-i}$, $i=1,\ldots,r$.

Consider $i=1$. Any element $j\in G/H$ is the class of an element $j\in G$, and $j$ can be written
$j=\al +\be$, where $\al \in \IZ_{p_1}\times \{0\}\times\cdots\times \{0\} = \IZ_{p_1}^{(1)}$ and
$\be \in \{0\}\times \IZ_{p_2}\times\cdots\times\IZ_{p_r}$. Let $\oG_1$ denote the image of 
$G_1 := \IZ_{p_1}^{(1)}$ in $G/H$, and let $\oF_1$ denote the image of $F_1 := \{0\}\times \IZ_{p_2}\times\cdots\times\IZ_{p_r}$
in $G/H$.

For all $\ell \in H^\perp$, we can rewrite \eqref{eq:circeqn} as
\begin{equation}\label{eq:circeqn1}
X_\ell = \frac{p_1}{k} \sum_{\obe\in \oF_1} \ep^{-\langle \ell,\be\rangle}\,\frac{1}{p_1}
\sum_{\oal\in \oG_1} \ep^{-\langle \ell,\al\rangle} Y_{\oal +\obe},
\end{equation}
where $\al, \be$ are any representatives of $\oal, \obe$ in the subgroups $G_1, F_1$ of $G$ (respectively),
and
\begin{align*}
\ep^{-\langle \ell,\al\rangle} &= \ep_{p_1}^{\ell_1(p_1-\al)},\\
\ep^{-\langle \ell,\be\rangle} &= \ep_{p_2}^{\ell_2(p_2-\be_2)}\cdots \ep_{p_r}^{\ell_r(p_r-\be_r)},
\end{align*}
where $\ell = (\ell_1,\ldots,\ell_r)$. (Again, since $\oal,\,\obe \in G/H$ and $\ell \in H^\perp$, $\ep^{-\langle \ell,\al\rangle}$
and $\ep^{-\langle \ell,\be\rangle}$ depend only on $\oal$ and $\obe$, respectively.)
We can reindex the $Y_{\oj = \oal +\obe}$,\, $\oj\in G/H$, as $Y_{\oal\obe}$,\, $\oal\in \oG_1$, $\obe\in \oF_1$.

Define
\begin{equation}\label{eq:circeqn2}
X_{\mu \obe}^{(1)} := \frac{1}{p_1} \sum_{\oal\in \oG_1} \ep_{p_1}^{\mu(p_1-\al)} Y_{\oal \obe}, \quad \mu \in \IZ_{p_1},\ \obe\in \oF_1;
\end{equation}
i.e., \eqref{eq:circeqn2} is a transformation corresponding to the inner sum in \eqref{eq:circeqn1}. If we order
the indices of the $Y_{\oal \obe}$ and $X_{\mu \obe}^{(1)}$ using lexicographic order of the pairs $(\obe, \oal)$ and
$(\obe, \mu)$ (respectively), taking the indices $\obe \in \oF_1$ in any order,
then the matrix of the linear transformation \eqref{eq:circeqn2} has block diagonal
form, with $k/p_1$ identical blocks given by \eqref{eq:circeqn2} with fixed $\obe$. 

\begin{remark}\label{rem:codim1} 
The transformation \eqref{eq:circeqn2} corresponds to the change of coordinates involved in rewriting in
standard product circular normal form (Example \ref{ex:standardcirc})
at a point $b$ of the codimension one stratum $T_1$. See \eqref{eq:circeqn3} and also \S\ref{subsec:codim1} below.
\end{remark}

Now, $H + G_1$ is a subgroup of $G$, and
\begin{align*}
(H + G_1)^\perp &= \{\ell\in H^\perp: \langle \ell, \al\rangle = 0, \ \text{for all } \al \in G_1\}\\
                          &= \{\ell\in H^\perp: \ell_1 = 0 \mod p_1\}.
\end{align*}

Consider the quotient $H^\perp / (H + G_1)^\perp$. Any element $\oell$ of this quotient is represented by
an element $\ell = (\ell_1,\ldots,\ell_r)\in H^\perp$, and $\ell_1 \in \IZ_{p_1}$ is uniquely determined by $\oell$.

\begin{lemma}\label{lem:Zp1}
\begin{equation*}
\frac{H^\perp}{(H +G_1)^\perp} \cong \IZ_{p_1}\,.
\end{equation*}
\end{lemma}

\begin{proof}
We recall that $G_1 = \IZ_{p_1}^{(1)}$ maps isomorphically onto the subgroup $\oG_1 = \IZ_{p_1}$ of $G/H$. 
The homomorphism $H^\perp \to \IZ_{p_1}$
given by $\ell \mapsto \ell_1$ induces an injection $H^\perp / (H + G_1)^\perp \hookrightarrow \IZ_{p_1} \hookrightarrow G/H$.
So it is enough to show that $|H^\perp / (H + G_1)^\perp| = p_1$. This is a consequence of the fact that $H^\perp \cong G/H$
and $(H+G_1)^\perp \cong G/(H+G_1) \cong (G/H)/\oG_1$. 
\end{proof}

Now, for any $\ell = (\ell_1,\ldots,\ell_r) \in H^\perp$, we can write $\ell = \mu + \nu$, where $\mu = (\ell_1,0,\ldots,0)$
and $\nu = (0,\ell_2,\ldots,\ell_r)$. We can reindex the $X_\ell$ as $X_{\mu \nu}$, so that the transformation
\eqref{eq:circeqn1} can be rewritten as the composite of \eqref{eq:circeqn2} with the transformation
\begin{equation}\label{eq:circeqn3}
X_{\mu \nu} = \frac{p_1}{k} \sum_{\obe\in \oF_1} \ep^{-\langle \ell,\be\rangle} X^{(1)}_{\mu \obe}.
\end{equation}
Since $\ell = (\mu,\nu) \in H^\perp$, $\ep^{-\langle \ell,\be\rangle}$ depends only on $\nu$ and $\obe$.
The change of coordinates \eqref{eq:circeqn3}
involves an invertible linear combination of the variables $X^{(1)}_{\mu \obe}$, $\obe\in \oF_1$,
for each fixed $\mu \in \IZ_{p_1}$.

Taking inverses of the transformations \eqref{eq:circeqn2} and \eqref{eq:circeqn3}, we can write
\begin{equation}\label{eq:circeqn4}
Y_{\oal \obe} = \sum_{\mu \in \IZ_{p_1}} \ep_{p_1}^{\al \mu} X_{\mu \obe}^{(1)} 
= \sum_{\mu \in \IZ_{p_1}} \ep_{p_1}^{\al \mu} \sum \ep^{\langle \be, \nu\rangle'} X_{\mu\nu}.
\end{equation}
where the latter sum is over all $\nu = (\nu_2,\ldots,\nu_r)$ such that $(\mu,\nu) \in H^\perp$, and
\begin{equation*}
\langle \be, \nu\rangle' := \sum_{i=2}^r \frac{k}{p_i}\be_i\nu_i, \quad \text{ so that } \quad  \ep^{\langle \be, \nu\rangle'} 
= \ep_{p_2}^{\be_2\nu_2}\cdots \ep_{p_r}^{\be_r\nu_r}.
\end{equation*}

We are now prepared to address the statement of Theorem \ref{thm:codim1}.

\subsection{Codimension one strata}\label{subsec:codim1}

\begin{proof}[Proof of Theorem \ref{thm:codim1}]
We continue to use the notation of \S\S \ref{subsec:groupcirc}, \ref{subsec:codim1combin}.
Consider 
\begin{equation*}
X_{\mu\nu} = w^{\ga_{\mu\nu}} x_{\mu\nu},
\end{equation*}
(where $X_{00} = x_{00} = z$), corresponding to Definition \ref{def:groupcircintro}. 
At a point of the codimension one stratum $T_1$, we can absorb the 
unit $w_2^{\ga_{\mu\nu}}\cdots w_r^{\ga_{\mu\nu}}$ into $x_{\mu\nu}$, so we have 
\begin{equation*}
X_{\mu\nu} = w_1^{h_{\mu\nu}/p_1} x_{\mu\nu},\quad \text{ where } h_{\mu\nu} \in \{0,\ldots,p_1-1\},
\end{equation*}
after this change of variable.

We recall that, at this point of $T_1$, $f(w,x,z) = f(w_1,x,z)$, $x = (x_{\mu\nu})$, has $k/p_1$ irreducible components.
We claim that these components are given by 
\begin{equation}\label{eq:prodcodim1}
\prod_{\oal\in\IZ_{p_1}} Y_{\oal\obe},\quad \text{ for each fixed } \obe.
\end{equation}
Indeed, for fixed $\obe$, the roots of $f$ corresponding to $Y_{\oal\obe}$, $\oal\in\IZ_{p_1}$, form an orbit of the
action of $\IZ_{p_1}$. It follows that \eqref{eq:prodcodim1} is $\IZ_{p_1}$-invariant and, therefore, a polynomial in $z$
with coefficients that are polynomials in $w_1$ and the $x_{\mu\nu}$, because the coefficients are given by the
elementary symmetric functions of the roots.

Fix $\obe$. Using \eqref{eq:circeqn2}, it is easy to check that, for each $\mu$, 
$v_1^{p_1-\mu} X_{\mu\obe}^{(1)}$
is invariant under the action of $\IZ_{p_1}$. Therefore,
\begin{equation*}
v_1^{p_1-\mu} X_{\mu\obe}^{(1)} = \eta_\mu (v_1^{p_1}, x),
\end{equation*}
where $\eta_\mu(w,x) \in \IC[w,x]$. It follows that, for each $\mu = 0,\ldots,p_1-1$,
$\eta_\mu (v_1^{p_1}, x)$ is
divisible by $v_1^{p_1}$, so that we can write
\begin{equation}\label{eq:coeffone}
X_{\mu\obe}^{(1)} = v_1^{p_1 m_\mu + \mu}\zeta_\mu (v_1^{p_1},x)
                      = w_1^{m_\mu + \mu/p_1} \zeta_\mu (w_1,x),
\end{equation}
where $\zeta_\mu (w_1,x) \in \IC[w_1,x]$ is not divisible by $w_1$.

It follows that $m_\mu =0$ and $\mu/p_1 = h_{\mu\nu}/p_1$,
for some $(\mu,\nu)$, and then that the $h_{\mu\nu}/p_1$ are independent of $\nu$, say, $h_{\mu\nu} = h_\mu$.
Therefore $(h_0,\ldots, h_{p_1-1}) = (0,\ldots,p_1-1)$ and
\begin{equation*}
\prod_{\oal\in\IZ_{p_1}} Y_{\oal\obe} = \De_{p_1} (y_{0\obe}, w_1^{1/p_1}y_{1\obe},\ldots w_1^{(p_1 -1)/p_1}y_{p_1-1,\obe}),
\end{equation*}
where
\begin{equation*}
y_{\mu \obe} = \sum \ep^{\langle \be, \nu\rangle'} x_{\mu\nu}.
\end{equation*}
(Compare with Lemma \ref{lem:prodk}.)
\end{proof}

\subsection{Reduction to group circulant normal form by blowing up}\label{subsec:redcirc}

In this subsection, we prove Theorem \ref{thm:ordpintro}. The proof uses several facts about the desingularization invariant 
$\inv$ of \cite{BMinv}, \cite{BMfunct};
in particular, to reduce to Theorem \ref{thm:normform} below. We can prove the latter following the structure of the
proof of \cite[Thm.\,1.22]{BBR},
and we refer to \cite[\S4.2]{BBR} for the basic details of the desingularization algorithm and the invariant $\inv$ that are needed. We
write $\inv(\nc(p))$ for the value of $\inv$ at a normal crossings point of order $p$; 
$\inv(\nc(p)) =(p,0,1,0,\ldots,1,0,\infty)$, where there are $p$ pairs in the sequence.

\begin{proof}[Proof of Theorem \ref{thm:ordpintro}]\renewcommand{\qedsymbol}{}
To begin the proof, we apply the desingularization algorithm (which transforms $X$ to a hypersurface 
before blowing up any
hypersurface points), and stop the algorithm when the maximum value of the invariant $\inv$ is $\leq \inv(\nc(p))$. If the
maximum value of $\inv$ is $\inv(\nc(p))$, then the maximum locus is a smooth subset $S$ of $X$ (i.e., of the strict
transform of $X$) of dimension $n-p+1$, transverse to the exceptional divisor $E$,
and we can blow up to eliminate any component of $S$ where $X$ is
not generically $\nc(p)$, and to guarantee that $X$ is $\nc(p)$ outside the intersection of $S$ with the exceptional divisor.

The first assertion of Theorem \ref{thm:ordpintro} including item (1) is then a consequence of Theorem \ref{thm:normform} following,
with $k=p$. (After applying this result, $X$ has order $<p$ in a deleted neighbourhood
of $S$, so the desingularization algorithm can be used to reduce to order $< p$ outside $S$.)
For the assertion (2) of Theorem \ref{thm:ordpintro} about a group action, see the Completion of the proof following
that of Theorem \ref{thm:normform} below.
\end{proof}

\begin{theorem}\label{thm:normform}
Consider an embedded hypersurface $X \hookrightarrow Z$. Let $U$ denote an open subset of $Z$.
Assume that (after an $\inv$-admissible sequence of blowings-up), the maximum value of $\inv$ on
$U$ is $\inv(\nc(k))$, so that the stratum $S := \{\inv = \inv(\nc(k))\}$ is a smooth subvariety of codimension $k$
in $U$, transverse to the exceptional divisor $E$. Suppose that $X$ is $\nc(k)$ on $S\backslash E$.
Then there is a finite sequence of $\inv_1$-admissible
blowings-up of $U$, preserving the $\nc(k)$-locus, after which $X$ has (product) group-circulant normal
form at evey point of the strict transform of $S$.

In terms of Definition \ref{def:groupcircintro}, we can assume, moreover, that the parameters $w_1,\ldots,\allowbreak w_r$
appearing in the formula for the product group-circulant singularity at any point of $S \cap E$ each correspond
to some component of $E$.
\end{theorem}

In this assertion, $\inv_1$ denotes the truncation of $\inv$ after the first of the pairs comprising $\inv$; see
\cite[\S4.2]{BBR}.

As in the case of \cite[Thm.\,1.22]{BBR}, we can divide the proof of Theorem \ref{thm:normform} into two parts. The first
part is a local version; it begins with the hypotheses of the Splitting Theorem \ref{thm:splitintro} and proves
Theorem \ref{thm:normform} in this setting. The second part is to show that the \emph{cleaning blowings-up}
involved in the proof following make sense as global admissible smooth blowings-up.
The global cleaning argument in the second part is identical to that in \cite[\S4.4]{BBR}, so we refer to the latter
and do not repeat the argument here; we present only the proof of the local part below.

For the local part, moreover, we give only the proof in the irreducible case below. The general case follows
exactly as in the proof of \cite[Thm.\,1.22]{BBR}, and we leave the details to the reader.

\begin{proof}[Proof of Theorem \ref{thm:normform}]
By Theorem \ref{thm:splitintro}, after finitely many $\inv$-admissible blowings-up with successive centres
given by the intersection of $S$ with a component of $E$,
we can assume that, at any given point $a$ of the stratum $S$, $X$ is defined by a regular
or analytic function $f(w,u,x,z)$ as in \eqref{eq:weierpoly},
which satisfies the conclusion of Theorem \ref{thm:splitintro}, 
in suitable coordinates $(w,u,x,z)$ in an \'etale or local analytic neighbourhood of $a$.
Let us assume that $f(w,u,x,z)$ is irreducible, so
we can follow the setup at the beginning
of Section \ref{sec:circ}.

We can again assume that $a_1(w,u,x)=0$, and we use the notation of \S\ref{subsec:circcomb}.
Define $Y_j$, $j=(j_1,\ldots,j_r) \in G$, as in \eqref{eq:factor}, and $X_\ell$, $\ell\in G$, as in \eqref{eq:circinv.1};
in particular, $X_0 = z$. If $j \in G$, then $Y_j$ depends only on the class $\oj$ of $j$ in $G/H$, so we write
$Y_{\oj} := Y_j$. Then, for all $\ell \in H^\perp$, we have
\begin{equation*}
X_\ell = \frac{1}{k}\sum_{\oj\in G/H} \ep^{-\langle \ell,j\rangle} Y_{\oj}\,,
\end{equation*}
where $\ep = \ep_k := e^{2\pi i/k}$. Note that, if $j$ is any representative of $\oj$ in $G$, then 
$\ep^{-\langle \ell,j\rangle} = \ep_{p_1}^{\ell_1(p_1-j_1)}\cdots \ep_{p_r}^{\ell_1(p_r-j_r)}$,
where $j = (j_1,\ldots,j_r)$ and $\ell = (\ell_1,\ldots,\ell_r)$, depends only $\oj$.

We describe the coefficient ideal of the marked ideal $(f,k)$ following the proof of \cite[Thm.\,4.1]{BBR},
and we refer to the latter for more details in a simpler situation.

Consider $\ell = (\ell_1,\ldots,\ell_r) \in H^\perp$. It is easy to see that each 
\begin{equation*}
v_1^{p_1-\ell_!}\cdots v_r^{p_r-\ell_r} X_\ell
\end{equation*}
is invariant under the action of $G = \IZ_{p_1}\times \cdots \times \IZ_{p_r}$. Therefore, 
\begin{equation*}
v_1^{p_1-\ell_!}\cdots v_r^{p_r-\ell_r} X_\ell = \eta_\ell (v_1^{p_1},\ldots v_r^{p_r}, u,x),
\end{equation*}
where $\eta_\ell(w,u,x) \in \IC\llb w,u,x\rrb$. It follows that, for each $\ell \in H^\perp \backslash \{0\}$, $\ell = (\ell_1,\ldots,\ell_r)$, 
$ \eta_\ell (v_1^{p_1},\ldots v_r^{p_r}, u,x)$ is
divisible by $v_i^{p_i}$, $i=1,\ldots,r$, so that we can write
\begin{equation}\label{eq:coeff}
\begin{aligned}
X_\ell &= v_1^{p_1 m_{\ell,1} + \ell_1}\cdots v_r^{p_r m_{\ell,r} + \ell_r} \zeta'_\ell (v_1^{p_1},\ldots,v_r^{p_r},u,x)\\
                      &= w_1^{m_{\ell,1} + \ell_1/p_1} \cdots w_r^{m_{\ell,r} + \ell_r/p_r} \zeta_\ell (w_1,\ldots,w_r,u,x),
\end{aligned}
\end{equation}
where $\zeta_\ell (w_1,\ldots,w_r,u,x) \in \IC\llb w,u,x\rrb$ is divisible by no $w_i$ (cf. \eqref{eq:coeffone}).

Write $k_i := k/p_i$,\, $i=1,\ldots,r$.
As in the proof of \cite[Thm. 4.1]{BBR}, the coefficient ideal of the marked ideal $(f,k)$ is equivalent to the marked ideal
$\ucC^1 = (\cC^1,k)$, where $\cC_1$ denotes the ideal generated by
$$
X_\ell^k  = w^{\ga_\ell} \zeta^k_\ell, \quad \ell \in H^\perp \backslash \{0\},
$$
on the maximal contact subspace $N^1 := \{z=0\}$, and $w^{\ga_\ell}$ is the monomial
\begin{equation*}
w^{\gamma_{\ell}} := w_1^{ k m_{\ell,1} + k_1 \ell_1} \cdots w_r^{k m_{\ell,r} + k_r \ell_r}.
\end{equation*}

Since $\inv(0) = (k,0,1,\ldots)$, there exists a smallest $\ga_\ell$, with respect to termwise order (i.e, $\ga \leq \ga'$ means
$\ga_i\leq \ga'_i$, $i=1,\ldots,r$, where $\ga = (\ga_1,\ldots,\ga_r),\, \ga' = (\ga'_1,\ldots,\ga'_r)$). In other words, 
\begin{equation*}\label{eq:mon1}
\al_1 := \min_{\ell\in H^\perp\backslash \{0\}} \gamma_{\ell}
\end{equation*}
is well-defined. 
Let $\ell^1$ denote a corresponding $\ell$ (realizing the minimum), and let $\xi_{\ell^1}$ denote the corresponding $\zeta_\ell$. 
Then $\xi_{\ell^1}$ has order $1$, since $\inv(0) = (k,0,1,\dots)$.
The monomial $w^{\al_1}$ generates
the monomial part of the coefficient ideal $\ucC^1$ (cf. \cite[\S A.6]{BMmin1}). 

It follows that the second coefficient ideal $\ucC^2$ (still with marked or associated order $k$), on the second
maximal contact subspace $N^2 := \{z=\xi_{\ell^1}=0\}$, is generated by
\begin{equation*}
w^{\gamma_\ell -\al_1} \zeta_{\ell}^k |_{N^2}\,,\quad \ell\neq \ell^1.
\end{equation*}
Therefore, for each $\ell\neq\ell^1$, we can write
\begin{equation*}
\zeta_\ell = \eta_\ell^{(1)} + w^{\beta_{\ell}^{(1)}}  \xi_\ell^{(1)},
\end{equation*}
where $\beta_{\ell}^{(1)} \in (1/k)\IN^r$,
$\eta_\ell^{(1)}$ is in the ideal generated by $\xi_{\ell^1}$, and $\xi_\ell^{(1)}|_{N^2}$ is not divisible by $w_i$, $i=1,\ldots,r$. 

Since $\inv(0) = (k,0,1,\ldots)$, there exists a smallest $\gamma_{\ell} + k\beta^{(1)}_{\ell}$, $\ell \neq \ell^1$; i.e.,
\begin{equation*}
\al_2 := \min_{\ell \neq \ell^1} \left(\gamma_{\ell} + k\beta_{\ell}^{(1)}\right)
\end{equation*}
is well-defined. Let
$\ell^2$ denote a corresponding $\ell$ (realizing the minimum), and let $\xi_{\ell^2}$ denote the corresponding $\xi_{\ell}^{(1)}$. 
Then $\xi_{\ell^2}$ has order $1$, since $\inv(a) = (k,0,1,0,1\dots)$. The monomial $w^{\al_2 -\al_1}$ generates the monomial part of the coefficient ideal $\ucC^2$. 

We can repeat this argument recursively for $3,\ldots,k-1$, to obtain exponents $\alpha_1,\ldots,\alpha_{k-1}$, as well as an ordering 
$\{0=\ell^0, \ell^1,\ldots,\ell^{k-1}\}$ of $H^\perp$. Note that, for each $j=1,\ldots,k-1$,
\begin{equation*}
\frac{1}{k}\alpha_{j} = \frac{1}{k}\alpha_{j-1} + \beta_{j} + \delta_{j},
\end{equation*}
where $\al_0 := 0$, $\beta_j = (\be_{j1},\ldots,\be_{jr})\in \IN^r$ and $\delta_j = (\delta_{j1},\ldots,\delta_{jr})$, with 
$\delta_{ji} \in \{0,1/p_i,\ldots,(p_i-1)/p_i\}$.

We can now apply a cleaning procedure as in the proof of \cite[Thm.\,4.1]{BBR}. We first blow up with codimension $1$
combinatorial centres $\{w_i =0\}$ ($1\leq i \leq r$), in the maximal
contact subspace $N^{k-1} = \{z=\zeta_{\ell^1}=\ldots= \zeta_{\ell^{k-2}}=0\}$ to reduce to $\beta_{k-1} =0$. 
We can continue this process successively in the
maximal contact subspaces $N^{k-2},\ldots,N^1$,
in order to get $\beta_j =0$ for all $j$. 

We can then make a formal (or \'etale or analytic) coordinate change
\begin{align*}
y_{\ell^1} &:= \xi_{\ell^1},\\
y_{\ell^j} &:= \eta_{\ell^j}^{(j-1)} +  \xi_{\ell^j}, \quad j=2,\dots,k-1,
\end{align*}
to reduce each $X_{\ell^j}$ to $w^{\alpha_j} y_{\ell^j}$,\, $j =1,\ldots,k-1$. We conclude that, after cleaning, $f$ can be written as
\begin{equation}\label{eq:delta}
\De_{G/H}\left(z, w^{\de_1} y_{\ell^1}, w^{\de_1 + \de_2} y_{\ell^2}, \ldots, w^{ \de_1+\cdots +\de_{k-1}} y_{\ell^{k-1}}\right), 
\end{equation}
where each $\de_j=(\de_{j1},\ldots,\de_{jr})$ with $\de_{ji} \in \frac{1}{p_i}\{0,\ldots,p_i-1\}$.

Note that \eqref{eq:delta} can be rewritten as 
\begin{equation}\label{eq:delta1}
\De_{G/H}\left(z, w^{\be_1 + \gamma_1} y_{\ell^1}, w^{\be_2 + \gamma_2} y_{\ell^2}, \ldots, w^{ \beta_{k-1} + \gamma_{k-1}} y_{\ell^{k-1}}\right), 
\end{equation}
where each $\be_j \in \IN^r$ and each $\ga_j=(\gamma_{j1},\ldots,\gamma_{jr})$, with $\gamma_{ji} \in \frac{1}{p_i}\{0,\ldots,p_i-1\}$.

We can now argue as in \cite[Remark 4.5]{BBR} to reduce each $\be_j$ in \eqref{eq:delta1} to $0$; i.e., we get the normal form
\begin{equation}\label{eq:delta2}
\De_{G/H} \left(X_{\ell^0}, X_{\ell^1}, \ldots, X_{\ell^{k-1}}\right),
\end{equation}
where
\begin{align*}
X_{\ell^0} &= z,\\
X_{\ell^j} &= w^{\ga_j} y_{\ell^j},\quad j=1,\ldots, k-1,
\end{align*}
with the preceding condition on the exponents $\ga_j$.
\end{proof}

\begin{proof}[Completion of the proof of Theorem \ref{thm:ordpintro}]
First, we obtain the smooth subset $S$ with the hypotheses of the Splitting 
Theorem \ref{thm:splitintro} using canonical resolution of singularities. All centres of blowing up involved are invariant
with respect to a group action, so that all exceptional divisors in the hypotheses of Theorem \ref{thm:splitintro} are
invariant under the action of $G$.
Since the normal crossings locus is also invariant with respect to a group action, it follows that the centres of all
blowings-up in Theorem \ref{thm:splitintro} are also invariant.

At the beginning of the proof of Theorem \ref{thm:normform}, therefore, the exceptional divisors (corresponding to the $\{w_i = 0\}$)
all invariant with respect to the group action. The group action takes maximal contact subvarieties to maximal contact
subvarieties, and it preserves the exceptional divisors $\{w_i = 0\}$, so the centres of blowing up involved in cleaning
are invariant with respect to the group action. 
Finally, at the end of the proof of Theorem \ref{thm:normform}, we argue as in \cite[Remark 4.5]{BBR} to reduce each
$\be_j$ to zero in \eqref{eq:delta1}  (i.e., to get group-circulant normal form), and the centres of blowing up
involved in the argument of \cite[Remark 4.5]{BBR} are also invariant.
\end{proof}

\section{Weighted blowing up of group-circulant singularities, and partial desingularization}\label{sec:wtblup}

In this section, we prove Theorem \ref{thm:orb1pintro} on partial desingularization. The theorem
involves weighted blowing up of the locus of 
group-circulant singularities of a given order. We use weighted blowings-up only in this situation.
A weighted blowing-up of an algebraic (or complex-analytic) variety $X$ provides an example of an 
orbifold, with covering charts (as in \eqref{eq:orbtriangle}) with finite abelian group actions, and a proper
birational (or bimeromorphic) morphism $\s: X' \to X$ induced by the quotients. We develop the relevant technology
of weighted blowing up and orbifolds only in the context needed.

We will work mainly in the algebraic case, but the methods apply to the complex-analytic case,
essentially using \cite[Ch.\,1]{AHV} on $\Specan$ (cf.\,\cite[\S1.6]{ABTW}). 
We will use the invariant for desingularization by weighted-blowings-up of Abramovich, Temkin and W{\l}odarczyk \cite{ATW19}, 
and it is convenient to use the formulation of W{\l}odarczyk \cite{Wlodar} in terms of cobordant blowings-up, which provide
an explicit birational morphism $\s: X' \to X$ without having to consider a stalk-theoretic quotient.

Weighted blowings-up of group-circulant singularities provide a good hands-on introduction to more general
techniques. In the complex-analytic case, for example, all arguments needed for Theorem \ref{thm:orb1pintro} 
can be carried out directly, without
using \cite{ATW19}, \cite{Wlodar}. We plan to present 
this in a follow-up article, in order to keep this paper of reasonable length.

The section is orgranized into three subsections. First, we recall the definition of weighted
blowing up of an affine space, together with the \'etale transition morphisms between the orbifold covering charts. Secondly,
we describe in detail the weighted blowing-up of a group-circulant singularity in local coordinates, and the way that the
construction enters into partial desingularization. (The special case of a cyclic group circulant singularity was treated
in Example \ref{ex:orbcpkintro}.) Finally, we prove Theorem \ref{thm:orb1pintro}. The combinatorial structure
of group-circulant singularities plays an important part. 

\subsection{Weighted blowing-up of an affine space}\label{subsec:wtblup}
In this subsection, we work with $\IC^n$ and the multiplicative group $\IC^*$, or, more generally, with
affine space $\IA^n$ over an (algebraically closed) field $\IK$ and the abelian multiplicative group (scheme)
$G_m$. 

\begin{definition}\label{def:wtproj}
Let $\om = (\om_1,\ldots,\om_n) \in \IN^n$ denote a \emph{weight vector}. The \emph{weighted projective space}
$\IP_{\om}^{n-1}$ is defined as the quotient or orbit space
\begin{equation*}
\IP_\om^{n-1} := (\IA^n\backslash\{0\}) / G_m
\end{equation*}
by the action of $G_m$ given by
\begin{equation*}
\la\cdot (y_1,\ldots,y_n) = (\la^{\om_1}y_1,\ldots,\la^{\om_n}y_n),
\end{equation*}
where $y=(y_1,\ldots,y_n)\in \IA^n$ and $\la \in G_m$.

Group actions are always understood to mean actions by automorphisms. Points of $\IP_\om^{n-1}$ are denoted
$[y_1,\ldots,y_n]$, where $[y_1,\ldots,y_n]$ means the equivalence class or orbit of $(y_1,\ldots,y_n) \in \IA^n\backslash\{0\}$.
\end{definition}

The weighted projective space $\IP_\om^{n-1}$ can be covered by $n$ affine charts
\begin{equation*}
U_i = \{[y_1,\ldots,y_n] \in \IP_\om^{n-1}: y_i \neq 0\},\quad i=1,\ldots,n.
\end{equation*}
Every point of $U_i$ can be written $[y_1,\ldots,1,\ldots,y_n]$ (with $1$ in the $i$th place).

For each $i$, let $A_i$ denote the affine quotient variety (or orbit space) $\IA^{n-1}/\mu_{\om_i}$,
where the action of the multiplicative cyclic group $\mu_{\om_i}$ on $\IA^{n-1}$ is given by
\begin{equation*}
\ep_i \cdot (y_1,\ldots,\hy_i,\ldots,y_n) = (\ep_i^{\om_1}y_1,\ldots,\widehat{\ep_i^{\om_i}y_i},\ldots,\ep_i^{\om_n}y_n),
\end{equation*}
and $\ep_i = \ep_{\om_i}$ denotes a primitive $\om_i$th root of unity. (A hat means that the corresponding entry is removed.)

\begin{remark}\label{rem:quot}
The preceding action of $\mu_{\om_i}$ corresponds to an action on $\IK[y_1,\ldots,\hy_i,\allowbreak \ldots,y_n]$ given by
$\ep_i\cdot y_j = \ep_i^{\om_j}y_j$, for each $j$, and 
\begin{equation*}
\IA^{n-1}/\mu_{\om_i} = \Spec\, \IK[y_1,\ldots,\hy_i,\ldots,y_n]^{\mu_{\om_i}},
\end{equation*}
where $\IK[y_1,\ldots,\hy_i,\ldots,y_n]^{\mu_{\om_i}}$ denotes the algebra of $\mu_{\om_i}$-invariant polynomials. The
affine quotient variety $A_i$ is determined by the ideal of relations among a set of generators (Hilbert basis)
for the algebra of invariant polynomials.
\end{remark}

It is easy to see that the mapping $U_i \to A_i$ given by taking $[y_1,\ldots,1,\ldots,y_n]$ to the orbit of
$(y_1,\ldots,\hy_i,\ldots,y_n)$ is a well-defined morphism (independent of the choice of $y_1,\ldots,\hy_i,\ldots,y_n$,
and defines an isomorphism $U_i \cong A_i$.

For each $i$, we thus have a morphism $\pi_i: \IA^{n-1} \to \IA^{n-1}/\mu_{\om_i} \cong U_i \subset \IP_\om^{n-1}$.
We call $\IA^{n-1}$ a \emph{covering (affine) chart} and $A_i \cong U_i$ a \emph{quotient chart}. The collection of
morphisms $\pi_i$, $i=1,\ldots,n$, gives the weighted projective space $\IP_\om^{n-1}$ the structure of an orbifold
(see below).

\begin{definition}\label{def:wtblup}
\emph{Weighted blowing-up of an affine space.} 
Let $Z \subset \IA^n \times \IP_\om^{n-1}$ denote the subvariety
\begin{equation*}
Z := \{(x,[y]) \in \IA^n \times \IP_\om^{n-1}: x \in \overline{[y]}\},
\end{equation*}
where $\overline{[y]}$ denotes the closure of the $G_m$-orbit $[y]$ of $y=(y_1,\ldots,y_n) \in \IA^n\backslash \{0\}$; i.e.
\begin{equation*}
Z := \{(x,[y]) \in \IA^n \times \IP_\om^{n-1}: x_j = t^{\om_j} y_j,\, j=1,\ldots,n,\, \text{ for some } t\in \IA^1\}.
\end{equation*}
The \emph{weighted blowing-up} of $W = \IA^n$ with (set-theoretic) \emph{centre} $\{0\}$ and weights $\om = (\om_1,\ldots,\om_n)$
corresponding to the affine coordinates $(x_1,\ldots,x_n)$,
is the morphism $\s: Z \to W = \IA^n$ induced by the projection $\IA^n \times \IP_\om^{n-1} \to \IA^n$ (or simply the variety $Z$). 
The weighted blowing-up $\s$
is an isomorphism over $\IA^n\backslash \{0\}$, and $E = \s^{-1}(0)$ is called the \emph{exceptional divisor}.
\end{definition}

The weighted blowiing-up $Z$ is covered by $n$ affine charts $V_1,\ldots,V_n$, where
\begin{equation*}
V_i = \{(x,[y]) \in Z: [y] \in U_i \subset \IP_\om^{n-1}\};
\end{equation*}
i.e., where $V_i$ is the set of points $(x,[y]) \in Z$ such that $y = [y_1,\ldots,1,\ldots,y_n]$ (with $1$ in the $i$th place),
and 
\begin{equation}\label{eq:wtblupmodel}
x_j = 
\begin{cases}
t^{\om_j} y_j, & j\neq i,\\
t^{\om_i},     & j=i,
\end{cases}
\end{equation}
for some $t \in \IA^1$.

For each $i$, there is a surjective morphism
\begin{align*}
\varphi_i : W_i = \IA^n &\to V_i\\
(y_1,\ldots,y_n) &\mapsto \left( (y_i^{\om_1}y_1,\ldots,y_i^{\om_i},\ldots,y_i^{\om_n}y_n), [y_1,\ldots,1,\ldots,y_n]\right).
\end{align*}
We will also write $W_i = W_{x_i}$.
In general, the morphism $\varphi_i$ is not injective; in fact, $\varphi_i(y) = \vp_i(z)$ if and only if $z$ belongs to the
orbit of $y$ under the action of $\mu_{\om_i}$ on $W_i$ given by $\xi\cdot y = z$, where
\begin{equation*}
z_i = \xi y_i \quad \text{ and }\quad z_j = \xi^{-\om_j} y_j,\,\, j \neq i,
\end{equation*}
for any $\xi \in \mu_{\om_i}$.
In other words, $V_i \cong W_i/\mu_{\om_i}$, for each $i$. Note that the action of the group $\mu_{\om_i}$ is free
outside the exceptional divisor $\vp_i^{-1}(E)$; therefore, the morphism $\vp_i$ is \'etale outside $\vp_i^{-1}(E)$.

\medskip
Lemma \ref{lem:trans} following asserts that the morphisms $\vp_i: W_i \to W_i/\mu_{\om_i}\cong V_i$, $i=1,\ldots,n$, provide 
what is called an \emph{orbifold structure}
for the weighted blowing-up $Z$, with covering charts $W_i$ and quotient charts $V_i$. We will sometimes
write $\tZ$ for the collection of covering charts, and $\tZ_\quot \cong Z$ for the quotient variety (cf. \eqref{eq:orbtriangle}).

\begin{remark}\label{rem:trans}
\emph{Transition between covering charts $W_i = W_{x_i},\, W_j = W_{x_j}$.} Say $i=1,\,j=2$. 
Let us re-label the coordinates $(y_1,\ldots,y_n)$ above as $(s,y_2,\ldots,y_n)$ in the case of $W_1$, and
$(z_1,t,z_3,\ldots,z_n)$ in the case of $W_2$. Then the transition between $W_1,\,W_2$ is given by the formulas
\begin{equation}\label{eq:trans}
\begin{aligned}
s^{\om_1} &= x_1 =  t^{\om_1}z_1,\\ 
s^{\om_2}y_2 &= x_2 = t^{\om_2},\\
s^{\om_j}y_j &= x_j = t^{\om_j}z_j, \quad j>2.
\end{aligned}
\end{equation}
It does not make sense to write $W_1\cap W_2$. The transition formulas determine a $\mu_{\om_2}$-invariant
open subset $W_{12}= W_{x_1x_2}$ of $W_2$, and a $\mu_{\om_1}$-invariant open subset $W_{21}= W_{x_2x_1}$ of $W_1$, corresponding
to $V_1 \cap V_2$ in $\tZ_\quot$. Clearly, $W_{12} = \{(z_1,t,z_3,\ldots,z_n) \in W_2: z_1\neq 0\}$
and $W_{21} = \{((s,y_2,\ldots,y_n) \in W_1: y_2\neq 0\}$ .The following lemma makes sense of the overlap of covering charts using
\'etale morphisms from a new covering chart.
\end{remark}

\begin{lemma}\label{lem:trans}
Given covering charts $W_i,\, W_j$ as above, there are
\begin{enumerate}
\item covering charts $\tW_{ji},\, \tW_{ij}$ with actions of 
$\mu_{\om_i}\times\mu_{\om_j},\,  \mu_{\om_j}\times\mu_{\om_i}$ (respectively) and an equivariant isomorphism $\tW_{ji} \cong \tW_{ij}$, where 
\begin{equation*}
\tW_{ji}/(\mu_{\om_i}\times\mu_{\om_j}) \cong \tW_{ij}/(\mu_{\om_j}\times\mu_{\om_i}) \cong V_i\cap V_j \subset Z;
\end{equation*}
\item \'etale morphisms
\begin{equation*}
\tW_{ji} \to W_i,\quad \tW_{ij} \to W_j,
\end{equation*}
commuting with the projections to $Z$ given by the group quotients, whose images are $W_{ji},\, W_{ij}$ (respectively).
\end{enumerate}
\end{lemma}

In particular, the \'etale morphisms in (2) take distinct orbits to distinct orbits. 
The collection of covering charts $W_i$ and $\tW_{ji}$ can
be completed to a finite \emph{covering altas} satisfying a natural cocycle condition (cf.\,\cite[{\S}II.1,\,Ex.\,2.12]{Hart}),
but we do not go into these details.

\begin{proof}[Proof of Lemma \ref{lem:trans}]
Say $i=1,\,j=2$. We use the notation of Remark \ref{rem:trans}. We define
\begin{align*}
\tW_{21} &:= \Spec\, \IK[s,y_2, y_2^{-1}, u_2, u_2^{-1}, y_3,\ldots, y_n] / (y_2 - u_2^{\om_2}),\\
\tW_{12} &:= \Spec\, \IK[z_1, z_1^{-1}, u_1, u_1^{-1}, t,z_3,\ldots, z_n] / (z_1 - u_1^{\om_1}).
\end{align*}
We also introduce the group actions on $\tW_{21}$ and $\tW_{12}$ given (using the generators $\ep_1$ of $\mu_{\om_1}$
and $\ep_2$ of $\mu_{\om_2}$) on $\tW_{21}$ by
\begin{align*}
\ep_1: (s, y_2, u_2, y_k) &\mapsto (\ep_1 s,\, \ep_1^{-\om_2} y_2,\, \ep_1^{-1} u_2,\, \ep_1^{-\om_k}y_k),\\
\ep_2: (s, y_2, u_2, y_k) &\mapsto (s,\, y_2,\, \ep_2 u_2,\,y_k),
\end{align*}
where $k\geq 3$, and on $\tW_{12}$ by
\begin{align*}
\ep_1: (z_1, u_1, t, z_k) &\mapsto (z_1,\, \ep_1 u_1,\, t,\, z_k),\\
\ep_2: (z_1, u_1, t, z_k) &\mapsto (\ep_2^{-\om_1}z_1,\, \ep_2^{-1} u_1,\, \ep_2 t,\, \ep_2^{-\om_k}z_k),
\end{align*}
where $k\geq 3$.

Then there is an equivariant isomorphism $\tW_{12} \to \tW_{21}$ induced by
\begin{equation*}
s = t u_1,\quad u_2 = u_1^{-1},\quad y_k = u_1^{-\om_k} z_k,\,\, k\geq 3,
\end{equation*}
and there are \'etale morphisms $\tW_{21} \to W_1$ and $\tW_{12} \to W_2$ induced by
$y_2 = u_2^{\om_2}$ and $z_1 = u_1^{\om_1}$, respectively. It is easy to see that the assertions of the lemma
are fulfilled.
\end{proof}

\begin{remark}\label{rem:gencentre}
More generally, the \emph{weighted blowing-up} of $\IA^m \times \IA^n$ with centre $\IA^m \times \{0\}$ and
weights $(\om_1,\ldots,\om_n)$ corresponding to the affine coordinates of $\IA^n$, is defined, following the definition
of the standard blowing-up, as the product of $\IA^m$ with the weighted blowing-up of $\IA^n$ as defined above.
\end{remark}

Theorem \ref{thm:orb1pintro} will involve global weighted blowings-up that are given locally (or \'etale locally) by
the affine construction above.

\subsection{Weighted blowing up of group-circulant singularities and their products}\label{subsec:wtblupcirc}
The main purpose of this subsection is to extend the weighted blowing up analysis of a circulant singularity
$\cp(k)$ in Example \ref{ex:orbcpkintro}, first to a product $\cp(k_1)\times\cdots\times\cp(k_s)$
(with a single divisor $w$), and secondly to a general (product) group-circulant singularity (Defintion \ref{def:groupcircintro}).
We recall that the latter induces a product of standard circulant singularities at every point of a codimension one stratum
(Remark \ref{rem:groupcircintro}).

With an eye towards the proof of Theorem \ref{thm:orb1pintro}, in the following examples we will compute also the invariant 
of \cite{ATW19}, which we denote $\ATWinv$ to distinguish it from the desingularization invariant $\inv$ of \cite{BMinv}, \cite{BMfunct},
used already in this article. As pointed out in \cite[\S1.5]{ATW19}, $\ATWinv$ is determined by the \emph{year-zero} $\inv$. At each point of
an embedded variety $X$ (or the cosupport of an ideal), the latter is a finite sequence $(b_1,0,b_2,0,\ldots, b_q,0,\infty)$, where $b_1$ is the order
and each $b_j$ is a positive rational number; let us write simply $\inv = (b_1,b_2,\ldots, b_q)$.

In fact, if the year-zero $\inv = (b_1,b_2,\ldots, b_q)$, then $\ATWinv = (a_1,a_2,\ldots,a_q)$, where
\begin{equation}\label{eq:invs}
a_j = b_1 b_2\cdots b_j,\quad j=1,\ldots,q;
\end{equation}
\begin{equation}\label{eq:invs1}
\text{i.e.,} \qquad b_1 = a_1\quad \text{and}\quad b_j =\frac{a_j}{a_{j-1}},\ \ j>1
\end{equation}
(cf.  \cite[Rmk.\,1.11]{BMinv}). This follows from the fact that $\inv$ and $\ATWinv$ can be defined using the same
sequences of local maximal contact hypersurfaces (given by part of a regular system of parameters, $x_1,\ldots,x_q$)
and corresponding coefficient ideals. The latter can be considered as marked ideals (each with an associated order)
which determine the successive entries of the invariant. The difference between $\inv$ and $\ATWinv$ is a certain
renormalization of the associate orders, corresponding to \eqref{eq:invs}; cf.\,\cite[\S5.1]{ATW19}. If $X$ is a hypersurface defined locally by
a function $f$, then the formal expansion of $f$ has a weighted homogeneous initial form with weights 
$(1/a_1,\dots, 1/a_q)$ associated to the parameters $(x_1,\ldots,x_q)$. (We are using the basic version of $\ATWinv$,
which does not distinguish divisorial from the remaining free parameters.)

In the language of \cite{ATW19}, \cite{Wlodar} and \cite{ABTW}, the centre of the weighted blowing-up is
given globally (on an open set where $(a_1,a_2,\ldots,a_q)$ is the maximum value of $\ATWinv$) by a \emph{Rees algebra},
which can be presented in local parameters as
\begin{equation*}
\cO_X[x_1t^{1/a_1}, x_2 t^{1/a_2},\ldots, x_q t^{1/a_q}]^{\mathrm{int}},
\end{equation*}
where $\mathrm{int}$ means integral closure in $\cO_X[t^{1/a}]$, with $a =$	 smallest positive rational number 
which is an integral multiple of $a_j,\, j=1,\ldots,q$.
Moreover, the weighted blowing-up can be described, following \cite{Wlodar}, by a \emph{cobordant blowing-up}
\begin{equation*}
B = \Spec_X \cO_X[t^{-1}, x_1t^{1/a_1}, x_2 t^{1/a_2},\ldots, x_q t^{1/a_q}]
\end{equation*}
(cf. Definition \ref{def:wtblup}).
We will not need to work explicitly with these notions, except for the fact that their existence implies that the weighted
blowings-up of group-circulant singularities that we use make sense as global weighted blowings-up, and
induce explicit birational morphisms of the orbifold quotients.

The fact that $\ATWinv$ is determined by $\inv$ may be used to simplify computations because $\inv$ and the maximal
contact hypersurfaces depend only on the equivalence class of a marked ideal (in the sense of \cite{BMfunct}). This is
useful (though not obligatory) in the examples following. We begin by adding the details above to Example \ref{ex:orbcpkintro}
on $\cp(k)$.

\begin{example}\label{ex:orbcpkintrobis}
We consider the cyclic \emph{circulant singularity} $\cp(k)$ defined by $f(w,x) = \De_k(x_0,w^{1/k}x_1,\ldots,w^{(k-1)/k}x_{k-1})$,
which we write in short form as $\De_k(w^{j/k}x_j)$. See Example \ref{ex:orbcpkintro}. We claim that, at any point where
$w=x=0$ (there may be additional parameters),
\begin{equation}\label{eq:invcpk}
\inv = \left(k,\,\frac{k+1}{k},\, 1,\, \frac{k}{k-1},\, \frac{k-1}{k-2},\,\ldots,\,\frac{3}{2}\right),
\end{equation}
with successive maximal contact hypersurfaces given by $x_0, w, x_1\, (\text{or } x_1, w), x_2,\ldots,x_{k-1}$.

To see this, we can use the fact that the marked ideal $(f,k)$ is equivalent to the monomial marked ideal
\begin{equation}\label{eq:monideal}
\left((x_0^k,\,wx_1^k,\,\ldots,\,w^{k-1}x_{k-1}^k),\, k\right).
\end{equation}
The equivalence is a consequence of the fact that the collection of $k$ factors of $\De_k(w^{j/k}x_j)$ (see \eqref{eq:circfact})
is an invertible linear combination of $w^{j/k}x_j$, $j=0,\ldots,k-1$ (see also \cite[Ex.\,A.13]{BMmin1}).

The calculation then goes as follows. We begin with maximal contact $x_0$ and the order $k$ as the first entry $b_1$ of $\inv$.
The first coefficient marked ideal is $((w^jx_j^k,\, j=1,\ldots,k-1),\,k)$, and the second entry $b_2$ of $\inv$ is $\ord(w^j x_j^k)/k 
= (k+1)/k$. We continue with the entries $b_j$ of $\inv$, maximal contact hypersurfaces and coefficient ideals as follows.
\begin{equation*}
\begin{aligned}
&b_2 = \frac{k+1}{k},\quad  &w,\qquad &(x_1^k,\,k)\,+\,(x_2^k,\,k-1)\, +\cdots +\,(x_{k-1}^k,\,2),\\
&b_3 = 1,\quad  &\ x_1,\qquad &(x_2^k,\,k-1)\, +\cdots +\,(x_{k-1}^k,\,2)\\
&              &        &= ((x_2^{k/(k-1)},\,x_3^{k/(k-2)},\ldots,x_{k-1}^{k/2}),\, 1),\\
&b_4 = \frac{k}{k-1},\quad   &x_2,\qquad &(x_3^{k/(k-2)},\ldots, x_{k-1}^{k/2},\, k/(k-1))\\
&              &        &= (x_3^{(k-1)/(k-2)},\ldots, x_{k-1}^{(k-1)/2},\, 1),
\end{aligned}
\end{equation*}
etc. For more details of how to make the computation, see the \emph{Crash course on the desingularizaion invariant},
\cite[Appendix]{BMmin1}. One could also make the computation (with a little more manipulation) directly from derivatives
of $f$, without passing to the equivalent monomial marked ideal.

It follows from \eqref{eq:invcpk} that
\begin{equation}\label{eq:ATWinvcpk}
\ATWinv = \left(k,\, k+1,\, k+1,\, \frac{k(k+1)}{k-1},\, \frac{k(k+1)}{k-2},\,\ldots,\, \frac{k(k+1)}{2}\right), 
\end{equation}
and the weight vector associated to the parameters
\begin{equation*}
(x_0,w,x_1,x_2,\ldots,x_{k-1})
\end{equation*}
is
\begin{equation*}
\frac{1}{k(k+1)}\left(k+1,\,k,\,k,\,k-1,\,k-2,\,\ldots,\,2\right).
\end{equation*}
This is the weight vector of Example \ref{ex:orbcpkintro}, scaled so that the total weighted degree of $\De_k$ is $1$.
\end{example}

\begin{example}\label{ex:orbprodcpk} Consider the \emph{product of cyclic circulant singularities}
$\cp(k_1)\times\cdots\times\cp(k_s)$,
where $\sum k_i = k$ (with a single divisor $w$), given by
\begin{equation*}
f(w,x) = \prod_{i=1}^s \De_{k_i}(w^{j/k_i}x_{ij},\,j=0,\ldots,k_i-1)
\end{equation*}
(using the short-hand notation of Example \ref{ex:orbcpkintro}),
in the closure $S = \{x_{ij} = 0: j=0,\ldots,k_i-1, i=1,\ldots,s\}$ of the $\nc(k)$-locus.
Suppose $k_1 = \max k_i$.
We first claim that, at a point where $w = x = 0$, 
\begin{equation}\label{eq:ATWinv}
\ATWinv = \left(k,\ldots,k,\, \frac{k(k_1+1)}{k_1},\, \left\{\frac{kk_i(k_1+1)}{k_1(k_i -j) + k_i}:\, j=1,\ldots,k_i-1,\, i=1,\ldots s\right\}\right),
\end{equation}
where the terms in the set should be written in ascending order, with a corresponding sequence of maximal contact hypersurfaces
given by 
\begin{equation*}
x_{i0},\, i=1,\ldots, s,\ \ w,\ \{x_{ij}:\, j=1,\ldots,k_i-1,\, i=1,\ldots s\},
\end{equation*}
in the same order as above.

The computation is similar to that of Example \ref{ex:orbcpkintrobis}. The marked ideal $(f,k)$ is equivalent to the 
monomial marked ideal
\begin{equation*}
\sum_{i=1}^k \left((w^j x_{ij}^{k_i}),\, k_i\right) \cong \left((w^{j/k_i} x_{ij}),\, 1\right)
\end{equation*}
The first $s$ terms of $\inv$ are $k,1,\ldots,1$, corresponding to maximal contacts $x_{10},\ldots,x_{s0}$,
and the corresponding coefficient ideal is
\begin{equation*}
\sum_{i=1}^k \left((w^j x_{ij}^{k_i})_{j\geq 1},\, k_i\right) \cong \left((w^{j/k_i} x_{ij})_{j\geq 1},\, 1\right)
\end{equation*}

Since $k_1 = \max k_i$, the next term of $\inv$ is $(k_1+1)/k_1$, and we can take maximal contact $w$. 
At this point, we have $(k,\ldots,k,\, k(k_1+1)/k_1,\ldots)$ in $\ATWinv$. The
next coefficient ideal is
\begin{equation*}
\sum_{i=1}^s\sum_{j=1}^{k_i-1}\left(x_{ij},\ \frac{k_1+1}{k_1} - \frac{j}{k_i} = \frac{k_1(k_i-j) + k_i}{k_1k_i}\right)
\cong \sum \left( x_{ij}^{\frac{k_1k_i}{k_1(k_i-j) + k_i}},\ 1\right). 
\end{equation*}
The next term of $\inv$ is the smallest value of
\begin{equation*}
\frac{k_1k_i}{k_1(k_i-j) + k_i},\quad j=1,\ldots,k_i-1,\ i=1,\ldots,s,
\end{equation*}
giving the smallest value of
\begin{equation*}
\frac{k(k_1+1)}{k_1}\frac{k_1k_i}{k_1(k_i-j) + k_i} = \frac{kk_i(k_1+1)}{k_1(k_i-j)+k_i},\quad j=1,\ldots,k_i-1,\ i=1,\ldots,s,
\end{equation*}
for the next term of $\ATWinv$.
At this step, the smallest value for $\inv$ is, in fact, $1$, with corresponding maximal
contacts $x_{i1}$, taken for all $i$ such that $k_i = k_1$. We continue in the same way to get \eqref{eq:ATWinv}.

Note that \eqref{eq:ATWinv} is given by the set of positive rational numbers
\begin{equation*}
\frac{k(k_1+1)}{k_1},\quad \frac{kk_i(k_1+1)}{k_1(k_i -j) + k_i},\,\, j=0,\ldots,k_i-1,\, i=1,\ldots s,
\end{equation*}
arranged in ascending order, with corresponding maximal contacts $w, x_{ij}$, in the same order. The corresponding weights,
associated to the parameters $w, x_{ij}$, are the reciprocals, given by $1/k(k_1+1)\,\times$
\begin{equation}\label{eq:prodwts}
k_1,\quad \frac{k_1(k_i -j) + k_i}{k_i}.
\end{equation}
We can multiply the weights \eqref{eq:prodwts} by $\ell/k_1$, where $\ell = \lcm\{k_i\}$, to get integer weights
\begin{equation}\label{eq:prodwts1}
\ell,\quad \ell - \left(j\frac{\ell}{k_i} - \frac{\ell}{k_1}\right)
\end{equation}
for the parameters $w, x_{ij}$, generalizing those of the $cp(k)$ case
in Example \ref{ex:orbcpkintro}. For example, if all $k_i = k_1$, then \eqref{eq:prodwts1} are the same weights as for $\cp(k_1)$.

\smallskip
Let us now analyze the weighted blowing-up $\s$ of $\cp(k_1)\times\cdots\times\cp(k_s)$, with weights \eqref{eq:prodwts1}
associated to the parameters $w, x_{ij}$, following the pattern of Example \ref{ex:orbcpkintro}.

In the orbifold $w$-chart, $\s$ is given by the substitution
\begin{align*}
w &= t^\ell,\\
x_{ij} &= t^{\ell - \left(j\frac{\ell}{k_i} - \frac{\ell}{k_1}\right)}\dx_{ij,}\quad j=0,\ldots,k_i-1,\, i=1,\ldots s, 
\end{align*}
and the group $\mu_\ell$ acts on the chart by
\begin{equation*}
\ep\cdot (t, \dx_{ij}) := (\ep t, \ep^{j\frac{\ell}{k_i} - \frac{\ell}{k_1}}\dx_{ij}), \quad \text{where } \ep \in \mu_\ell.
\end{equation*}
The group action is free outside the exceptional divisor $\{t=0\}$.

Likewise, the orbifold $x_j$-chart has an action of $\mu_{k-j+1}$, $j=0,\ldots,k-1$.

The strict transform by $\s$ of the stratum $S = \{x=0\}$ intersects only
the $w$-chart (as an invariant smooth subvariety), and
the pullback of $\cp(k_1)\times\cdots\times\cp(k_s)$ is given in this chart by 
\begin{equation*}
t^{k(\ell +\ell/k_1)}\prod\De_{k_i}(\dx_{i0},\ldots,\dx_{i,k_i-1}), 
\end{equation*}
which is normal crossings.

As in Example \ref{ex:orbcpkintro}, we obtain a Hilbert basis
\begin{align*}
W &= t^\ell,\\
X_{i0} &= t^{\ell/k_1}\dx_{i0},\quad 1\leq i\leq s,\\
X_{i1} &= \dx_{i1},\quad \text{if }k_i =k_1,\\
X_{ij} &= t^{\ell - \left(j\frac{\ell}{k_i} - \frac{\ell}{k_1}\right)}\dx_{ij}, \quad \text{otherwise},\\
S_{\mu,\la} = S_{\mu,(\la_{ij})} &= t^\mu \prod_{i=1}^s\prod_{j=0}^{k_i-1} \dx_{ij}^{\la_{ij}},
\end{align*}
where $\mu$ and the $\la_{ij}$ are nonnegative integers such that $\la_{i1} = 0$ if $k_i = k_1$ and
\begin{equation*}
\mu + \sum_{i=1}^s\la_{i0}\left(\ell - \frac{\ell}{k_1}\right) + \sum_{i=1}^s\sum_{j=1}^{k_i-1}\la_{ij}\left( j\frac{\ell}{k_i} - \frac{\ell}{k_1} \right) = \nu \ell, \quad \nu = \nu_{\mu,\la} \geq 1.
\end{equation*}
(The latter invariants include those preceding.)
Note that $\displaystyle{0 < \left( j\frac{\ell}{k_i} - \frac{\ell}{k_1} \right) < \ell}$, if $j=1$ and $k_i < k_1$, or $j>1$.

The relations among the invariants include
\begin{equation*}
\begin{aligned}
\prod_{i,j} X_{ij}^{\la_{ij}} = &= t^{\sum_i\la_{i0}\frac{\ell}{k_1} + \sum_i\sum_{j\geq 1}\la_{ij}\left(\ell - \left(j\frac{\ell}{k_i} - \frac{\ell}{k_1}\right)\right)}\prod_{i,j} \dx_{ij}^{\la_{ij}}\\
&= t^{\sum_i \la_{i0}\ell + \sum_i\sum_{j\geq 1}\la_{ij} \ell - \left(\mu + \sum_i \la_{i0}\left(\ell - \frac{\ell}{k_1}\right) 
+ \sum_i\sum_{j\geq 1}\la_{ij}\left(j\frac{\ell}{k_i} - \frac{\ell}{k_1}\right)\right)} t^\mu \prod_{i,j} \dx_{ij}^{\la_{ij}}\\
&= W^{\sum \la_{ij} - \nu} S_{\mu,\la}.
\end{aligned}
\end{equation*}

Exactly as in Example \ref{ex:orbcpkintro}, the image of $\left\{\prod\De_{k_i}(\dx_{i0},\ldots,\dx_{i,k_i-1}) = 0\right\}$
by the quotient morphism (given by the Hilbert basis) is just $\cp(k_1)\times\cdots\times\cp(k_s)$, so that the quotient
variety is the intersection of the product circulant hypersurface with the toric variety defined by the ideal of relations.
We call this quotient an \emph{orbifold product circulant singularity}. 

The induced morphism from the quotient is given by $w=W,\, x_{i0} = WX_{i0},\, x_{i1} = WX_{i1} \text{ if } k_i = k_1$,
and $x_{ij} = X_{ij}$ for all other $i,j$.
The analogue of Remark \ref{rem:orbcpkintro} then also holds exactly in the same way as the latter. In particular, the preceding
toric variety is resolved by a single blowing-up, after which we recover the $\cp(k_1)\times\cdots\times\cp(k_s)$ singularity.

Note that the cyclic groups acting on the various orbifold charts have order bounded by $\ell +1$.
\end{example}

If there are several divisors $w_1,\ldots,w_r$, we have to repeat the above construction for $i=1,\ldots,r$,
as in the general group-circulant case following.

\begin{example}\label{ex:orbgencirc}
\emph{Group-circulant singularity} $\De_\Ga$ (Definition \ref{def:groupcircintro}).
The group $\Ga = G/H$, as in Definition \ref{def:groupcircintro}, where
$G = \IZ_{p_1}\times \cdots \times \IZ_{p_r}$, corresponding to parameters $w = (w_1,\ldots,w_r)$ representing
exceptional divisors. Each $p_i | k$, where $k = |G/H|$. On a stratum 
\begin{equation*}
T_i = \{x=w_i=0,\, w_h \neq 0, h\neq i\}
\end{equation*}
of codimension $1$ in $S = \{x=0\}$, $\De_{G/H}$ induces standard product circulant normal form
$\cp(p_i)\times\cdots\times\cp(p_i)$ (with $k/p_i$ factors).

We perform $r$ weighted blowings-up with successive centres $S \cap \{w_i = 0\} = \{w_i \allowbreak = x = 0\}$, $i=1,\ldots,r$
(where the $w_i$ can be ordered according to history), and where the parameters $w_i,\,x_j$, for fixed $i$, have associated
weights given by Example \ref{ex:orbprodcpk} for $\cp(p_i)\times\cdots\times\cp(p_i)$. 
Each such centre intersects only the $w_1\cdots w_{i-1}$-chart. 
Then $S$ lifts to the $w_1\cdots w_r$-chart of the composite of weighted blowings-up,
with group $\mu_{p_1}\times\cdots\times \mu_{p_r}$. The action of the group is free outside $\{y_1\cdots y_r = 0\}$,
and the given group-circulant singularity has strict transform $\nc(k)$ in the $w_1\cdots w_r$-chart.

Note that the desingularization algorithms of \cite{BMfunct}, \cite{ATW19}, \cite{Wlodar} are functorial with respect
to field extension, and each of the weighted blowings-up above are given by Example \ref{ex:orbprodcpk}
for $\cp(p_i) \times \cp(p_i)$, over the field $\overline{\IK(w_{i+1},\ldots,w_r)}$. For example, for $i=1$, extension of the
field $\IK$ by $\overline{\IK(w_2,\ldots,w_r)}$ corresponds to the fibre product $\times_{\Spec\,\IK}\, \Spec\,\overline{\IK(w_2,\ldots,w_r)}$,
by which the group-circulant singularity at the given point becomes $\cp(p_i)\times\cdots\times\cp(p_i)$, and $\inv$
(or $\ATWinv$) becomes the invariant for $\cp(p_i)\times\cdots\times\cp(p_i)$.

The successive weighted blowings-up are nontrivial only over the $w_1\cdots w_i$-charts, $i=1,\ldots,r-1$.
The group acting on the $w_1\cdots w_ix_{ij}$-chart is $\mu_{p_1}\times \cdots \times \mu_{p_i} \times \mu_{p_{i+1} - j+1}$.
Following the pattern of Examples \ref{ex:orbcpkintro} and \ref{ex:orbprodcpk}, we can write a Hilbert basis and relations
for the action of $\mu_{p_1}\times \cdots \times \mu_{p_r}$ on the $w_1\cdots w_r$-chart. Then,
following Remark \ref{rem:orbcpkintro}, we can also recover the original group-circulant hypersurface from the 
orbifold quotient, by blowing up the codimension 1 strata one at a time.
The final orbifold groups have order bounded by $p_1 +\cdots + p_r +1
\leq k(n-k+1)+1$ (since $r \leq n-k+1$).

The more general case of a product group-circulant singularity can be handled exactly in the same way.
\end{example}

\subsection{Partial desingularization preserving normal crossings}\label{subsec:pardes} 
We are finally in a position to complete the proof of the Partial Desingularization Theorem \ref{thm:orb1pintro}.

The proof is by induction on the integer
$p$ in the theorem, and provides a sequence of weighted blowings-up. The only blowings-up with non-trivial
weights that are needed are those for product group-circulant singularities, as described in Example \ref{ex:orbgencirc}. 
A standard (classical) blowing-up is a weighted blowing-up with trivial weights (all weights $=1$). The successive
(weighted) blowings-up (needed, for example, in the inductive step) are defined in the orbifold covering charts,
and we have to show that these blowings-up make sense globally; in particular, that they induce birational (or
bimeromorphic) morphisms of the orbifold quotient varieties.

For this purpose, it is very convenient to use the Abramovich-Temkin-W{\l}odarczyk construction of weighted blowings-up
\cite{ATW19} and W{\l}odarczyk's formulation in terms of cobordant blowings-up \cite{Wlodar}. In this construction, $\ATWinv$
is constant on the centre of a weighted blowing-up, which can be defined globally by a Rees algebra.

For an embedded variety $X\hookrightarrow Z$, where $Z$ is smooth, a sequence of cobordant blowings-up is a sequence
of morphisms of varieties $Z = Z_0 \leftarrow Z_1 \leftarrow \cdots \leftarrow Z_t = Z'$, where each $Z_j$ has an action of 
a torus $T_j = G_m^j$ with finite stabilizer subgroups and geometric quotient (i.e., the fibres are single orbits). The centres
of the successive blowings-up are $T_j$-invariant, and there is an induced sequence of birational morphism
$Z_{j-1}/T_{j-1} \leftarrow Z_j/T_j$ commuting with the quotient morphisms. The construction is functorial with respect
any group action on $X\hookrightarrow Z$.

The successive cobordant blowings-up induce weighted blowings-up of the orbifold covering charts (Definition \ref{def:wtblup}
and \eqref{eq:wtblupmodel} is a model), and the quotients $Z_j/T_j$ are covered by (\'etale morphisms
from) the orbifold quotient charts; cf.\,\cite[{\S}A.6]{ABTW}, \cite[\S3.4]{ATW19}.

\begin{proof}[Proof of Theorem \ref{thm:orb1pintro}]
We begin by applying Theorem \ref{thm:ordpintro} to obtain a morphism $\s_p: X_p \to X$ given by a finite sequence
of admissible blowings-up (equivariant with respect to an
action of a group $G$), preserving the locus $X^{\nc(p)}$ of normal crossings points of order $\leq p$, such that
$X_p$ is a hypersurface and has maximum order $\leq p$, the subset $S_p$ of points of order $p$ of $X_p$ is smooth
and of dimension $n-p+1$ (unless empty), where $n = \dim X$, and $X_p$ has only (product) group-circulant singularities
of order $p$ at every point of $S_p$ (in particular, $X_p$ is generically $\nc(p)$ on $S_p$).

We can assume, moreover, that $X_p \subset Z_p$ is an embedded hypersurface ($Z_p$ smooth), and
(as in Theorem \ref{thm:ordpintro})
that the parameters $w_i$ appearing in the formula
for the product group-circulant singularity at any non-$\nc(p)$ point of $S_p$ (Definition \ref{def:groupcircintro}) each
correspond to some component $D_j = D_{j(i)}$ of the exceptional divisor (which is a simple
normal crossings divisor in $Z_p$, transverse to $S_p$). Let $D$ denote the sum of these components $D_j$
(say, $j = 1,\ldots,q$); $X_p$ is $\nc(p)$ on $S_p\backslash D$.

We claim that the weighted blowings-up of the (product) group circulant singularities in $S_p$, as given in
Example \ref{ex:orbgencirc}, make sense as a sequence of global weighted blowings-up with centres which have 
support $S_p \cap D_j$, $j=1,\ldots,q$ (where the $D_j$ can be ordered according to history), after which $\tX_p$ 
has only normal crossings in a deleted neighbourhood of $\tS_p$ (notation of Theorem \ref{thm:orb1pintro}).
This claim will be formalized as Lemma \ref{lem:globalcirc} following.

After applying the lemma, $\tX_p$  has only normal crossings in a deleted neighbourhood of $\tS_p$. We can then
apply the induction hypothesis in the complement of $\tS_p$; the centres of blowing up will be isolated from 
$\tS_p$ (and therefore have closed supports) because $\tX_p$ is already normal crossings of order $<p$ in a deleted
neighbourhood of $\tS_p$. The theorem follows.
\end{proof}

As an example which is pertinent to the proof of Lemma \ref{lem:globalcirc}, let us consider the case that $X_p$ has a
$\cp(k_1) \times\cdots\times \cp(k_s)$ singularity at a point $a\in S_p$ (where $k_1 +\cdots +k_s = p$); cf. Example
\ref{ex:orbprodcpk}. Recall that, in the latter, the set of product circulant points is given by the maximum
locus of $\ATWinv$ (where $\ATWinv$ takes the value \eqref{eq:ATWinv}). Since $\ATWinv$ and the 
desingularization algorithms are functorial with respect to \'etale morphisms, the maximum stratum of $\ATWinv$
in $X_p$, in a neighbourhood of $a$, lies in $S_p$ and corresponds to that in Example \ref{ex:orbprodcpk}.

The weighted blowing-up corresponding to the maximum stratum of $\ATWinv$ is described locally
by parameters that define a sequence of maximal contact hypersurfaces, together with associated weights coming from
$\ATWinv$. If we choose such parameters $(v, y_{ij})$ at $a$, appropriately ordered (cf. Example \ref{ex:orbprodcpk}, where
$v$ is a local generator of the ideal of the exceptional divisor, corresponding to $w$ in the example),
then the weighted blowing-up induces that of Example \ref{ex:orbprodcpk} in an appropriate \'etale neighbourhood of $a$.

\smallskip
In general, consider an snc divisor $E$ on a smooth variety $Z$, with components $E_1,\ldots,E_t$, say. Let $E_1,\ldots,E_s$
denote the components of $E$ at a point $a\in Z$, and let $(v_1,\ldots,v_s,\, z_1, z_2,\ldots)$ denote a system of parameters
for $Z$ at $a$, where each $v_j$ is a local generator of the ideal of $E_j$. We call $v_1,\ldots,v_s$ \emph{divisorial parameters},
and the remaining coordinates $z_1,z_2,\ldots$ \emph{free parameters}. If $v_1',\ldots,v_s'$ is another
set of local generators of the ideals of $D_1,\ldots,D_s$ (respectively), then each $v_j'$ is a unit times $v_j$, and the change
of parameters induces an isomorphism of the fields of fractions $\IK(v) = \IK(v_1,\ldots,v_s)$ and $\IK(v')$. The free coordinates
induce regular parameters at $a$ in $Z\times_{\Spec\,\IK} \Spec\,\overline{\IK(v)}$, where $\overline{\IK(v)}$ denotes an 
algebraic closure of $\IK(v)$.

The field of fractions $\IK(v)$ is isomorphic to the field of fractions of the ring of polynomials in local generators at $a$ of the ideals
of all components $E_1,\ldots,E_t$ of $E$, so we will also write $\IK(E) = \IK(E_1,\ldots,E_t)$ for the isomorphism class of $\IK(v)$.

\begin{lemma}\label{lem:globalcirc}
We use the notation of the proof of Theorem \ref{thm:orb1pintro} above, so that $X_p$ has a given product circulant singularity
$\cp(k_1^i)\times\cdots\cp(k_{s_i}^i)$, where $\sum_\ell k_\ell^i = p$, at every point of a codimension one stratum
$S_p \cap \left(D_i \backslash \cup_{j\neq i} D_j\right)$. For each $i=1,\ldots,q$, let $D^{(i)}$ denote the snc divisor given
by all components of $D$ \emph{except} $D_i$. Then $X_p$ is transformed to a variety that is $\nc(p)$ on (the strict fransform of) $S_p$,
by a sequence of $q$ weighted blowings-up whose successive weighted centres have support $S_p \cap D_i$, and correspond
to the maximum value of $\ATWinv$ on $S_p$ in $X_p \times_{\Spec\,\IK} \Spec\,\overline{\IK(D^{(i)})}$. (Note that this maximum value
is the value of $\ATWinv$ for $\cp(k_1^i)\times\cdots\cp(k_{s_i}^i)$. The $D_i$ can be ordered according to history.)
\end{lemma}

\begin{proof}
Consider $S_p \cap D_i$, for some $i$. At a point $a \in S_p \cap D_i$ (outside the corresponding codimension one stratum),
we consider a local generator $v_i$ of the ideal of $D_i$, together with free parameters for $Z_p$ at $a$, all as free parameters
(so that only local generators of the ideals of $D_j$, $j\neq i$, are divisorial).

Say $i=1$. Suppose $D_1,\ldots,D_s$ are the components of $D$ at $a$. After extension by the field $\overline{\IK(v_2,\ldots,v_s)}$,
$X_p$ has a $\cp(k_1^1)\times\cdots\cp(k_{s_1}^1)$ singularity at $a$, and successive maximal contact hypersurfaces at $a$
for the corresponding value of $\ATWinv$ are obtained by appropriate derivatives in the free parameters, which are induced by
those for $Z_p$. 

If $a$ is a (product) group-circulant point as in Example \ref{ex:orbgencirc} , then extension by the field $\overline{\IK(v_2,\ldots,v_s)}$
pulls back to extension by $\overline{\IK(w_2,\ldots,w_s)}$ in the example.
We thus obtain a global weighted centre of blowing up in $X_p$, corresponding to the weighted blowing-up
in Example \ref{ex:orbgencirc} for $i=1$.

We continue in the same way, for $i=2,\ldots,q$.
\end{proof}

\begin{proof}[Proof of Addendum \ref{add:orb1pintro}]
The additional blowings-up needed for the orbifold group-circulant singularities in $S_p$ have centres each given by
the intersection of $S_p$ with a component of the distinguished divisor (where the components are ordered
by history, for example). This is a consequence of Remark \ref{rem:orbcpkintro} and the analogous statements, in 
general (see Examples \ref{ex:orbprodcpk} and \ref{ex:orbgencirc}).

For $S_{p-1}$, the analogous blowings-up have centres in the complement of $S_p$, but these centres are 
closed in $\tX_p$ because $\tX_p$ has only normal crossings singularities in a deleted neighbourhood of $\tS_p$.
Likewise, for all $q<p$.
\end{proof}

\section{Quotient of normal crossings by a finite abelian group action}\label{sec:orb}
The purpose of this final section is to show that the singularities which appear in the Partial Desingularization 
Theorem \ref{thm:orb1pintro} in addition to neighbours of group-circulant singularities
can be described explicitly by formulas analogous to group-circulant for quotients by a finite abelian group action of an 
invariant normal-crossings
ideal; see Example \ref{ex:invtnc} and Theorem \ref{thm:invtnc}.

\subsection{Action of a finite abelian group}\label{subsec:gpaction}
Given a multiplicative cyclic group $\mu_p$, we write $\ep_p$ for the generator $e^{2\pi i/p}$ of $\mu_p$.
Any finite abelian group $G$ is isomorphic to a product $\mu_{p_1} \times \cdots \times \mu_{p_r}$ of cyclic groups.
If $G = \mu_{p_1} \times \cdots \times \mu_{p_r}$, let $\mu_{p_i}^{(i)}$ denote the subgroup 
$\{1\} \times \cdots \times \mu_{p_i} \times \cdots \times \{1\}$, and let $\ep_{p_i}^{(i)}$ denote the generator of $\mu_{p_i}^{(i)}$
corresponding to $\ep_{p_i}$.

We recall that any linear representation of a finite abelian group $G$, over an algebraically closed field, is diagonalizable; 
equivalently, every irreducible linear representation of $G$ has dimension $1$.

Let $Z$ denote a smooth variety over an algebraically closed field (or a smooth complex analytic variety), and suppose that
$G = \mu_{p_1} \times \cdots \times \mu_{p_r}$ acts on $Z$. Let $a \in Z$ denote a fixed point
of the action of $G$. Then $G$ also acts as a group of automorphisms of the local ring $\cO_{Z,a}$. It follows that there is a system
of regular (or analytic) coordinates $x = (x_1,\ldots,x_n)$ at $a=0$ with respect to which the action of $G$ is diagonal; i.e.,
\begin{equation}\label{eq:diagcoords}
\ep_{p_i}^{(i)}\cdot x_j = \ep_{p_i}^{\ga_{ij}}x_j,\quad \text{where }\ \ 0\leq \ga_{ij} < p_i,
\end{equation}
for every $i=1,\ldots,r$ and $j=1,\ldots, n$. (In the algebraic case, there is an equivariant \'etale morphism from an affine 
neighbourhood of $a$ to the tangent space of $Z$ at $a$, by \cite[Lemme fondamental]{Luna}; see also \cite [Sect.\,5]{BKS}.)

We say that $f\in \cO_{Z,a}$ is \emph{$G$-semi-invariant} if 
\begin{equation}\label{eq:semiinv}
g\cdot f = \prod_{i=1}^r \ep_{p_i}^{\ga_i(g,f)} f,
\end{equation}
for every $g\in G$. (Given $g\in G$, we say also that $f$ is \emph{$g$-semi-invariant} if it satisfies \eqref{eq:semiinv}.)
If an ideal $I \subset \cO_{Z,a}$ has a system of $G$-semi-invariant generators, then
$I$ is $G$-invariant; i.e., $g\cdot I = I$, for every $g\in G$.

\begin{lemma}\label{lem:semiinvgens}
Suppose $G = \mu_{p_1} \times \cdots \times \mu_{p_r}$ acts on $Z$, and $a\in Z$ is a fixed point.
Let $I \subset \cO_{Z,a}$ denote a $G$-invariant ideal. Then $I$ has a system of $G$-semi-invariant generators.
\end{lemma}

\begin{proof}
Let $x = (x_1,\ldots,x_n)$ denote coordinates at $a=0$ with respect to which the action of $G$ is diagonal,
as in \eqref{eq:diagcoords}. Fix an element $f \in I$, and let $G_0$ denote a subgroup of $G$ such that $f$ is $G_0$-semi-invariant. 
(The argument following will be used iteratively, beginning with $G_0 = \{1\}$.) Let $\hf_0(x)$ denote the formal expansion
$f$ at $a=0$, with respect to the coordinates $x$.

Let $p := \lcm(p_1,\ldots,p_r)$. Given $g\in G$, we can write
\begin{equation*}
\hf_0(x) = \sum_{j=0}^{p-1} h_j(x),\quad \text{where}\  \ (g\cdot h_j)(x) = \ep_p^j h_j(x),\ \ j=0,\ldots,p-1
\end{equation*}
(in a unique way). Note that each $h_j$ is also $G_0$-semi-invariant since the action of $G$ is diagonal with respect to
$(x_1,\ldots,x_n)$ and $f$ is $G_0$-semi-invariant.

Moreover, since $I$ is $G$-invariant, it follows that
\begin{equation*}
g^\ell \cdot \hf_0(x) = \sum_{j=0}^{p-1} \ep_p^{\ell j}h_j(x),\quad \ell=0,\ldots,p-1.
\end{equation*}
The matrix $(\ep_p^{\ell j})$ is invertible (Vandermonde matrix), so that $h_j(x)$ belongs to the formal ideal generated by $I$,
for each $j$. But each $g^\ell\cdot f \in \cO_{Z,a}$, so also every $h_j \in I \subset \cO_{Z,a}$. Each $h_j$ is both $g$-semi-invariant
and $G_0$-semi-invariant.

Now, given a system of generators $f_1,\ldots,f_t$ of $I$, we can apply the above argument to each $f_i$, $i=1,\ldots,t$, and each
$g \in G$, one at a time, to end up with a system of $G$-invariant generators of $I$. 
\end{proof}

\begin{corollary}\label{cor:semiinvgens}
Suppose $G = \mu_{p_1} \times \cdots \times \mu_{p_r}$ acts on $Z$, and $a\in Z$ is a fixed point.
Let $E\subset Z$ denote
an snc divisor, and $S$ a smooth subvariety of $Z$ that is snc with respect of $E$. Assume that the ideals $I_{E_j} \subset \cO_{Z,a}$ of 
all components $E_1,\ldots,E_\al$ of $E$ at $a$, and the ideal $I_S \subset \cO_{Z,a}$ of $S$ are all $G$-invariant. Then $Z$ admits a 
$G$-diagonal coordinate system $x = (x_1,\ldots,x_n)$ at $a=0$ which is adapted to $E$ and $S$; i.e., each $I_{E_j} = (x_{i_j})$,
for some $i_j$, and also $I_S$ is generated by some of the $x_i$.
\end{corollary}

\begin{proof}
First assume that $E$ and $S$ are transverse. 
Fix a coordinate system $x=(x_1,\ldots,x_n)$ with respect to which the action of $G$ is diagonal. By Lemma \ref{lem:semiinvgens}, 
we can assume that each $I_{E_j}$ has a $G$-semi-invariant generator $h_j$, and also that $I_S$ has $G$-semi-invariant
generators $h_{\al+1},\ldots,h_\be$. The result follows by replacing $\be$ of the coordinates $x_i$ (say, $x_1,\ldots,x_{\be}$)
by $h_1,\ldots,h_\be$.

In general, $E = E' \cup E''$, where $E'$ is transverse to $S$ and the components of $E''$ each contain $S$. Say
that the components of $E'$ are $E'_1,\ldots,E'_\al$. Apply the preceding case to $E'$ and $S$ to find a $G$-diagonal coordinate
system $(w,y,z) = (w_1,\ldots,w_\al, y_1,\ldots,y_\be, z_1,\ldots, z_{n-\al-\be})$ such that $I_{E''_j} = (w_j),\, j=1,\ldots,\al$, and
$I_S = (y_1,\ldots,y_\be)$.

Let $E''_1,\ldots,E''_\ga$ denote the components of $E''$; say, $I_{E''_j} = (h_j)$, $j=1,\ldots,\ga$. We can assume
that $I_S = (h_1,\ldots,h_\ga,y_{\ga +1},\ldots, y_\be)$. (The case $\ga \geq \be$ is trivial.)
Note that $I_{S'} := (y_{\ga +1},\ldots, y_\be)$ is $G$-invariant.
The result follows from the transverse case applied to $E$ and $S'$.
\end{proof}

\begin{example}\label{ex:invtnc}
\emph{Invariant normal crossings ideal.} In this example, we study a normal crossings ideal $I \subset \cO_{Z,a}$
which is invariant with respect to a cyclic group $G = \mu_p$ of automorphisms
of $\cO_{Z,a}$ (where we use the notation above). The case of a general finite abelian group $G$ will be treated
in Theorem \ref{thm:invtnc} below, but we begin with the simpler case here to try to make the argument 
in the general case clearer.

Let $I \subset \cO_{Z,a}$ denote a principle product ideal $I = (f)$, where $f = f_1\cdots f_k$ in $\cO_{Z,a}$ and the
gradients of the factors $f_1,\ldots,f_k$ at $a$ are linearly independent; i.e., $I$ is normal crossings. Assume that $I$ is
$G$-invariant, but that $I$ does not split (nontrivially) into $G$-invariant factors (i.e., we cannot write $f = \vp_1\vp_2$,
neither factor a unit, where $(\vp_i)$ is $G$-invariant, $i=1,2$). Since $I$ is $G$-invariant, the ideal 
$J = I_S = (f_1,\ldots,f_k) \subset \cO_{Z,a}$ 
of the normal crossings locus $S = \{f_1 = \cdots = f_k = 0\}$ is also $G$-invariant.

We claim that $Z$ admits a coordinate system
$y = (y_0,\ldots, y_{n-1})$ at $a$ with respect to which $G$ is diagonal, and $I$ is generated by
\begin{equation}\label{eq:invtnc}
\prod_{j=0}^{k-1} \left( y_0 + \ep_k^j y_1 + \ep_k^{2j} y_2 + \cdots + \ep_k^{(k-1)j} y_{k-1}\right) = \De_k(y_0,\ldots,y_{k-1}).
\end{equation}

If there is also a divisor $E\subset Z$ such that $E,\,S$ are transverse and each $I_{E_i}$ is $G$-invariant,
then we take the coordinates $(y_0,\ldots,y_{n-1})$ adapted to $E$ (as well as $S$).

More generally, In the case that $I$ splits into $G$-invariant factors, it follows that we can choose
$G$-diagonal coordinates in which $I$ is generated by a product of functions of the form \eqref{eq:invtnc}.
If each component $(f_j)$ is $G$-invariant, then this result is given simply by Lemma \ref{lem:semiinvgens}
and Corollary \ref{cor:semiinvgens}.

\medskip
We begin with a $G$-diagonal coordinate system $x=(x_1,\ldots,x_n)$ (adapted to $E$ and $S$). In particular,
for each $i=1,\ldots,n$, $\ep_p \cdot x_i = \ep_p^{q(i)} x_i$, for some $q(i) = 0,\ldots,p-1$.

Since $I$ does not split into $G$-invariant factors, there is a homomorphism of $G = \mu_p$ onto a cyclic subgroup
$\mu_k$ of the group of permutations of $f_1,\ldots,f_k$. The kernel of this homomorphism is the cyclic group
$H = \mu_{p/k}$ generated by $\ep_p^k = \ep_{p/k}$. The kernel $H$ is the subgroup of $G$ of all elements
which induce automorphisms of $(f_j)$, for any $j$. 

By Lemma \ref{lem:semiinvgens}, we can assume that
each $f_j$ is $H$-semi-invariant; i.e., $\ep_{p/k}\cdot f_j = \ep_{p/k}^{\ga(j)} f_j$, for some $\ga(j) = 0, \ldots p/k - 1$.
This means that every monomial $x^\al = x_1^{\al_1}\cdots x_n^{\al_n}$ in the formal expansion of $f_j$ has
the property $\ep_{p/k}\cdot x^\al = \ep_{p/k}^{\ga(j)} x^\al$. But
\begin{equation*}
\ep_{p/k}\cdot x^\al = \ep_p^k\cdot x^\al = \ep_p^{k\sum_i q(i)\al_i} = \ep_{p/k}^{\sum q(i)\al_i};
\end{equation*}
i.e., $\sum q(i) \al_i \equiv \ga(j) \mod p/k$.

In particular, for every monomial $x^\al$ in the formal expansion of $f_1$, $\sum q(i)\al_i$ takes only
the values $\ga(1),\, \ga(1)+p/k,\ldots, \ga(1)+(k-1)p/k \mod p$.

We can write $f_1$ formally as a sum of $G$-semi-invariant terms
$f_1 = f_{1,0} + \cdots f_{1,p-1}$, where $\ep_p \cdot f_{1,m} = \ep_p^m f_{1,m}$, $m=0,\ldots,p-1$; i.e., for every
monomial $x^\al$ in the formal expansion of $f_{1,m}$, we have $\sum q(i)\al_i \equiv m \mod p$.
It follows that
\begin{equation*}
f_1 = h_0 + \ldots + h_{k-1},\quad \text{where each}\ \ h_\ell = f_{1,\ga(1) + \ell p/k},
\end{equation*}
so that
\begin{equation*}
\ep_p\cdot h_\ell = \ep_p^{\ga(1) + \ell p/k} \cdot h_\ell = \ep_p^{\ga(1)} \ep_k^\ell h_\ell.
\end{equation*}

Now, for each $j=0,\ldots,k-1$, we have
\begin{equation}\label{eq:diagvand}
\ep_p^j\cdot f_1 = \ep_p^{j\ga(1)} \sum_{\ell=0}^{k-1} \ep_k^{j\ell} h_\ell.
\end{equation}
Since the matrix $(\ep_k^{j\ell})$ is invertible, it follows that each $h_\ell \in \cO_{Z,a}$, and the gradients
of $h_0,\ldots,h_{k-1}$ at $a$ are linearly independent. The ideal $I$ is generated by the product $f_0'\cdots f_{k-1}'$,
where $f_j' = \ep_p^j\cdot f_1$, $j=0,\ldots,k-1$, 
so we obtain \eqref{eq:invtnc} with $y_\ell = h_\ell$, $\ell = 0,\ldots, k-1$.
\end{example}

The following proposition corresponds to Example \ref{ex:invtnc} for the action of a general finite abelian
group $G$, which we can take to be $G = \mu_{p_1}\times\cdots\times\mu_{p_r}$, for some $p_1,\ldots,p_r$.
A precise formula for the coefficient matrix $\big(\ep_p^{\be_{j \ell}}\big)$ in \eqref{eq:Ginvtnc} below,
generalizing the Vandermonde matrix $(\ep_k^{j \ell})$ in \eqref{eq:invtnc}, will be given in the proof;
see \eqref{eq:nccoeffmatrix}. The general coefficient matrix can be expressed as a block matrix with nested blocks of the form
diagonal $\times$ Vandermonde analogous to the linear transformation in the right-hand side of \eqref{eq:diagvand};
cf. Example \ref{ex:gencirc}.

\begin{theorem}\label{thm:invtnc}
Let $G = \mu_{p_1}\times\cdots\times\mu_{p_r}$ denote a group of automorphisms of $\cO_{Z,a}$ (where we use
the notation above), and let $I \subset \cO_{Z,a}$ denote a principle product ideal $I = (f)$, 
where $f = f_1\cdots f_k$ in $\cO_{Z,a}$ and the
gradients of the factors $f_1,\ldots,f_k$ at $a$ are linearly independent; i.e., $I$ is normal crossings. Assume that $I$ is
$G$-invariant, but that $I$ does not split (nontrivially) into $G$-invariant factors. Furthermore, let $E\subset Z$ 
denote a divisor such that each $I_{E_i}$ is $G$-invariant, and $E,\,S$ are transverse,
where $S$ is the normal crossings locus $S = \{f_1 = \cdots = f_k = 0\}$.

Then $Z$ admits a coordinate system
$y = (y_0,\ldots, y_{n-1})$ at $a$, adapted to $E$ (as well as $S$), with respect to which $G$ is diagonal, 
such that $I$ is generated by
\begin{equation}\label{eq:Ginvtnc}
\prod_{j=0}^{k-1} \left( y_0 + \ep_k^{\be{j,1}} y_1 + \cdots + \ep_k^{\be_{j,k-1}} y_{k-1}\right),
\end{equation}
where the coefficient matrix $\big(\ep_k^{\be_{j \ell}}\big)$ is given by the formula \eqref{eq:ncgencoeffmatrix}
in the proof.
\end{theorem}

\begin{remark}
Roughly speaking, the coefficient matrix is a nested block matrix, nested by $i=1,\ldots,s$, where $s\leq r$ and the $i$'th
nested blocks are are of the form diagonal\,$\times\,(\ep_{q_i}^{j \ell})$ (the latter is the $q_i \times q_i$ Vandermonde),
where $q_i < p_i$ and $q_1\cdots q_s = k$. The corresponding diagonal part is also involved in \eqref{eq:diagvand},
but the constants in the diagonal in \eqref{eq:diagvand} do not affect the ideal $I$.
The formula \eqref{eq:Ginvtnc} should be compared to
$\De_{G/H}$, where $H$ is the subgroup of $G$ of automorphisms preserving any (and therefore all)
of the normal crossings factors, where the circulant matrix corresponding to the coordinate system $y$
is the general nested block form, as in Example \ref{ex:gencirc}.
\end{remark}

\begin{proof}[Proof of Theorem \ref{thm:invtnc}]
We begin with a $G$-diagonal coordinate system $x=(x_1,\ldots,x_n)$, adapted to $E$ and $S$. In particular,
for each $i=1,\ldots,r$ and $j=1,\ldots,n$, 
\begin{equation*}
\ep_{p_i}^{(i)} \cdot x_j = \ep_{p_i}^{\ga_{ij}} x_j,\quad \text{for some}\ \ \ga_{ij} = 0,\ldots,p_i-1.
\end{equation*}
Thus, for any monomial $x^\al = x_1^{\al_1}\cdots x_n^{\al_n} \in \cO_{Z,a}$, 
\begin{equation*}
\ep_{p_i}^{(i)} \cdot x^\al = \ep_{p_i}^{\sum_j \ga_{ij}\al_j} x_\al.
\end{equation*}

Let $H$ denote the subgroup of $G$ of all elements that induce automorphisms of the principle ideal $(f_1)$.
Since $I = (f)$ does not split into $G$-invariant factors, $H$ is also the subgroup of $G$ of elements that induce
automorphisms of $(f_j)$, for any $j$, and $|G/H| = k$. We can assume that each $f_j$ is $H$-semi-invariant,
by Lemma \ref{lem:semiinvgens}.

Consider the surjective group homomorphism
\begin{equation}\label{eq:gphomo}
\mu_{p_1} \to \frac{\mu_{p_1}^{(1)}+H}{H} \subset \frac{G}{H}.
\end{equation}
The quotient group $(\mu_{p_1}^{(1)}+H)/H \cong \mu_{q_1}$, where $q_1$ is the smallest $q$ such that
$\big(\ep_{p_1}^{(1)}\big)^q \in H$; i.e., such that $\big(\ep_{p_1}^{(1)}\big)^q$ induces an automorphism of $(f_1)$.
Let $K_1$ denote the kernel of the homomorphism \eqref{eq:gphomo}. Then
\begin{equation*}
K_1 \cong \mu_{p_1/q_1}\quad \text{and}\quad K_1 \cong \mu_{p_1}^{(1)} \cap H.
\end{equation*}

For each $j=1,\ldots,k$, $f_j$ is $H$-semi-invariant and, therefore, $K_1$-semi-invariant, so we can write
\begin{equation*}
\ep_{p_1/q_1}^{(1)} \cdot x^\al = \ep_{p_1/q_1}^{\ga_1(j)} x^\al,\quad \text{where}\ \  0 \leq \ga_1(j) < p_1/q_1,
\end{equation*}
for every monomial $x^\al$ in the formal expansion of $f_j$ (where $\ep_{p_1/q_1}^{(1)}$ means $\big(\ep_{p_1}^{(1)}\big)^{q_1})$.

We can write $f_1$ formally as a sum of $\mu_{p_1}^{(1)}$-semi-invariant terms
$f_1 = f_{1,0} + \cdots f_{1,p_1-1}$, where $\ep_{p_1}^{(1)} \cdot f_{1,\ell} = \ep_{p_1}^\ell f_{1,\ell}$, $\ell=0,\ldots,p_1-1$; i.e., every
monomial $x^\al$ in the formal expansion of $f_{1,\ell}$ satisfies $\ep_{p_1}^{(1)} \cdot x^\al = \ep_{p_1}^\ell x^\al$,
so that $\sum \ga_{1,j}\al_j \equiv \ell \mod p_1$. But $\sum_{j=1}^n \ga_{1,j}\al_j = \ga_1(1) \mod p_1/q_1$, for every 
monomial $x^\al$ in the formal expansion of $f_1$.
Therefore, for every monomial $x^\al$ in the formal expansion of $f_1$, $\sum \ga_{1,j}\al_j$ takes only
the values $\ga_1(1),\, \ga_1(1)+p_1/q_1,\ldots, \ga_1(1)+(q_1-1)p_1/q_1 \mod p_1$.

In other words,
\begin{equation*}
f_1 = h_0 + \ldots + h_{q_1-1},\quad \text{where each}\ \ h_\ell = f_{1,\ga_1(1) + \ell p_1/q_1},
\end{equation*}
so that
\begin{equation*}
\ep_{p_1}^{(1)}\cdot h_\ell = \ep_{p_1}^{\ga_1(1) + \ell p_1/q_1}  h_\ell = \ep_{p_1}^{\ga_1(1)} \ep_{q_1}^\ell h_\ell,
\end{equation*}
and, in general,
\begin{equation*}
\left(\begin{matrix}
f_1\\[.1em]
\ep_{p_1}^{(1)}\cdot f_1\\[.1em]
\vdots\\[.1em]
\big(\ep_{p_1}^{(1)}\big)^{q_1-1}\cdot f_1
\end{matrix}\right)
= \left(\begin{matrix}
1 & & & \\[.1em]
& \ep_{p_1}^{\ga_1(1)} & & \\[.1em]
& & \ddots & \\[.1em]
& & & \ep_{p_1}^{(q_1 - 1)\ga_1(1)}
\end{matrix}\right)
\left(\begin{matrix}
1 & 1 & \cdots & 1 \\[.1em]
1 & \ep_{q_1} & \cdots & \ep_{q_1}^{q_1-1} \\[.1em]
& & \ddots & \\[.1em]
1 & \ep_{q_1}^{q_1-1} & \cdots & \ep_{q_1}^{(q_1-1)^2}
\end{matrix}\right)
\left(\begin{matrix}
h_0 \\[.1em]
h_1 \\[.1em]
\vdots \\[.1em]
h_{q_1-1}
\end{matrix}\right)
\end{equation*}

Each $h_\ell$ is $\ep_{p_1}^{(1)}$-semi-invariant and $H$-semi-invariant; therefore, $H_1$-semi-invariant,
where
\begin{equation*}
H_1 := \mu_{p_1}^{(1)} + H.
\end{equation*}

Now we can consider the surjective group homomorphism
\begin{equation}\label{eq:gphomo1}
\mu_{p_2} \to \frac{\mu_{p_2}^{(2)}+H_1}{H_1} \subset \frac{G}{H_1}.
\end{equation}
The quotient group $(\mu_{p_2}^{(2)}+H_1)/H_1  \cong \mu_{q_2}$, where $q_2$ is the smallest $q$ such that
$\big(\ep_{p_2}^{(2)}\big)^q \in H_1$. Then $q_2$ divides $k/q_1$.
(If $q_1q_2 =k$, we will be finished at this step.) Let $K_2$ denote the kernel of the group homomorphism \eqref{eq:gphomo1}.
Then $K_2 \cong \mu_{p_2/q_2}$ and $K_2 \cong \mu_{p_2}^{(2)} \cap H_1$.

Each $h_\ell$ is $K_2$-semi-invariant, so we can write 
\begin{equation*}
\ep_{p_2/q_2}^{(2)} \cdot h_\ell = \ep_{p_2/q_2}^{\ga_2(\ell)} h_\ell.
\end{equation*}

Following the previous pattern, we can write
\begin{equation*}
h_{\ell_1} = \sum_{\ell_2 = 0}^{q_2 -1} h_{\ell_1\ell_2},\quad \ell_1=0,\ldots,q_1-1, \ \ell_2 = 0,\ldots, q_2-1
\end{equation*}
(formally), where each $h_{\ell_1\ell_2}$ is $\ep_{p_2}^{(2)}$-semi-invariant and, in fact,
\begin{equation*}
\ep_{p_2}^{(2)}\cdot h_{\ell _1\ell_2}= \ep_{p_2}^{\ga_2(\ell_1) + \ell_2 p_2/q_2}  h_{\ell_1\ell_2} 
= \ep_{p_2}^{\ga_2(\ell_1)} \ep_{q_2}^{\ell_2} h_{\ell_1\ell_2}.
\end{equation*}
So, for each $m_1 = 0,\ldots,q_1-1$, and $m_2 = 0,\ldots,q_2-1$, we get
\begin{equation}\label{eq:nccoeffmatrix}
\begin{aligned}
\big(\ep_{p_2}^{(2)}\big)^{m_2} \big(\ep_{p_1}^{(1)}\big)^{m_1} \cdot f_1
&= \big(\ep_{p_2}^{(2)}\big)^{m_2} \cdot \left(\ep_{p_1}^{m_1\ga_1(1)} \sum_{\ell_1=0}^{q_1-1} \ep_{q_1}^{m_1\ell_1} h_{\ell_1}\right)\\
&= \ep_{p_1}^{m_1\ga_1(1)} \sum_{\ell_1=0}^{q_1-1} \ep_{q_1}^{m_1\ell_1} \big(\ep_{p_2}^{(2)}\big)^{m_2} \cdot h_{\ell_1}\\
&= \ep_{p_1}^{m_1\ga_1(1)} \sum_{\ell_1=0}^{q_1-1} \ep_{q_1}^{m_1\ell_1} \ep_{p_2}^{m_2\ga_2(\ell_1)}
 \sum_{\ell_2 = 0}^{q_2 -1} \ep_{q_2}^{m_2\ell_2} h_{\ell_1\ell_2}.
\end{aligned}
\end{equation}
The coefficient matrix here is a nested block matrix with blocks of the form diagonal\,$\times$\,Vandermonde.

If $k=q_1q_2$, then we are finished at this step; otherwise, after at most $r$ steps since $I = (f)$ cannot be factored
into nontrivial $G$-invariant factors. If we finish after $s$ steps (where $s\leq r$), then $k = q_1 q_2 \cdots q_s$ where
$q_1,\ldots,q_s$ are introduced successively as above, and we get
\begin{multline}\label{eq:ncgencoeffmatrix}
\big(\ep_{p_s}^{(s)}\big)^{m_s} \cdots \big(\ep_{p_1}^{(1)}\big)^{m_1} \cdot f_1\\
= \ep_{p_1}^{m_1\ga_1(1)} \sum_{\ell_1=0}^{q_1-1} \ep_{q_1}^{m_1\ell_1} \ep_{p_2}^{m_2\ga_2(\ell_1)}
 \sum_{\ell_2 = 0}^{q_2 -1} \cdots\, \ep_{q_{s-1}}^{m_{s-1}\ell_{s-1}} \ep_{p_s}^{m_s\ga_s(\ell_{s-1})}\sum_{\ell_s = 0}^{q_s -1} \ep_{q_s}^{m_s\ell_s} h_{\ell_1\ell_2\cdots \ell_s},
\end{multline}
for all $m_i = 0,\ldots,q_i -1$, $i =1,\ldots,s$.

Since the coefficient matrix is invertible, it follows (as in Example \ref{ex:invtnc} above) that the
$h_{\ell_1\cdots\ell_s} \in \cO_{Z,a}$, and their gradients are linearly independent. The ideal $I$ is generated by the product
of the $\big(\ep_{p_s}^{(s)}\big)^{m_s}\cdots \big(\ep_{p_1}^{(1)}\big)^{m_1} \cdot f_1$,
so we obtain \eqref{eq:Ginvtnc} with the $y$-coordinates given by the $h_{\ell_1\cdots\ell_s}$.
\end{proof}

\bibliographystyle{amsplain}

\end{document}